\documentclass[12pt,reqno]{amsart}
\usepackage{amsmath,amssymb,amsthm,mathtools,calc,verbatim,enumitem,tikz,url,hyperref,mathrsfs,cite,fullpage}
\usepackage{bbm}
\usepackage{stmaryrd}
\usepackage{cleveref}
\usepackage{amssymb}
\usepackage{amsthm}
\usepackage{amsmath}
\usepackage{graphicx}
\usepackage{marvosym}
\usepackage{empheq}
\usepackage{latexsym}
\usepackage{fontenc}
\usepackage{color}

\newtheorem{theorem}{Theorem}[section]
\newtheorem{lemma}[theorem]{Lemma}
\newtheorem{corollary}[theorem]{Corollary}
\newtheorem{conjecture}[theorem]{Conjecture}
\newtheorem{observation}[theorem]{Observation}

\newtheorem{claim}[theorem]{Claim}
\newtheorem*{claim*}{Claim}

\theoremstyle{definition}
\newtheorem{definition}[theorem]{Definition}

\newtheorem*{qu*}{Question}
\theoremstyle{remark}

\newcommand\N{\mathbb{N}}

\newcommand\E{\operatorname{\mathbb{E}}}
\newcommand\cA{\mathcal{A}}

\newcommand\cD{\mathcal{D}}
\newcommand\cE{\mathcal{E}}
\newcommand\cF{\mathcal{F}}
\newcommand\cG{\mathcal{G}}
\newcommand\cH{\mathcal{H}}
\newcommand\cI{\mathcal{I}}

\newcommand\cK{\mathcal{K}}
\newcommand\cL{\mathcal{L}}
\newcommand\cM{\mathcal{M}}
\newcommand\cN{\mathcal{N}}

\newcommand\cR{\mathcal{R}}
\newcommand\cS{\mathcal{S}}
\newcommand\cT{\mathcal{T}}

\renewcommand\Pr{\operatorname{\mathbb{P}}}

\newcommand\eps{\varepsilon}
\renewcommand\leq{\leqslant}
\renewcommand\geq{\geqslant}
\renewcommand\le{\leqslant}
\renewcommand\ge{\geqslant}
\renewcommand\to{\rightarrow}

\numberwithin{equation}{section}

	\def\eps{\varepsilon}

	\def\cX{\mathcal{X}}

	\def\Prob{\mathbb{P}}

	\def\Var{\mathrm{Var}}

	\def\N{\mathbb{N}}

	\def\cA{\mathcal{A}}
        \def\cW{\mathcal{W}}

	\def\<{\langle }
	\def\>{\rangle }

\begin{document}

\title{The Chromatic Number of Very Dense\\ Random Graphs}
\author{Zhifei Yan}

\address{IMPA, Estrada Dona Castorina 110, Jardim Bot\^anico,
Rio de Janeiro, 22460-320, Brazil}\email{zhifei.yan@impa.br}

\thanks{}

\begin{abstract}
The chromatic number of a very dense random graph $G(n,p)$, with $p \ge 1 - n^{-c}$ for some constant $c > 0$, was first studied by Surya and Warnke, who 
conjectured that the typical deviation of $\chi(G(n,p))$ from its mean is of order $\sqrt{\mu_r}$, 
where $\mu_r$ is the expected number of
independent sets of size $r$, and $r$ is maximal such that $\mu_r > 1$, except when $\mu_r = O(\log n)$. They moreover proved their conjecture in the case $n^{-2} \ll 1 - p = O(n^{-1})$. 

In this paper, we study $\chi(G(n,p))$ in the range $n^{-1}\log n \ll 1 - p \ll n^{-2/3}$, that is, when the largest independent set of $G(n,p)$ is typically of size 3. We prove in this case that $\chi(G(n,p))$ is concentrated on some interval of length $O(\sqrt{\mu_3})$, 
and for sufficiently `smooth' functions $p = p(n)$, that there are infinitely many values of $n$ such that $\chi(G(n,p))$ is not concentrated on any interval of size $o(\sqrt{\mu_3})$. We also show that $\chi(G(n,p))$ satisfies a central limit theorem in the range $n^{-1} \log n \ll 1 - p \ll n^{-7/9}$.
\end{abstract}
	
	\maketitle 
\section{Introduction}

Understanding the distribution of the chromatic number of the random graph $G(n,p)$ is one of the most fascinating (and notorious) problems in the study of random graphs. For constant $p \in (0,1)$, the asymptotic value of $\chi(G(n,p))$ was determined in 1987 by Bollob{\'a}s~\cite{B}, who proved that
\begin{align*}
\chi(G(n,p))=\bigg(\frac{1}{2}+o(1)\bigg)\frac{n}{{\log_{1/(1-p)}n}}   
\end{align*}
with high probability. Shamir and Spencer~\cite{SS} moreover showed that $\chi(G(n,p))$ is concentrated on an (unknown) interval of size $O(\sqrt{n})$. Their bound was later improved by a factor of $\log n$, by Alon (see~\cite[Section~7.9, Exercise 3]{Probabilisticmethod}).   


For very sparse random graphs, there has been much more progress over the past few decades. First, Shamir and Spencer~\cite{SS} showed that $\chi(G(n,p))$ is concentrated on a bounded number of values when $p < n^{-1/2-\varepsilon}$, and Alon and  Krivelevich \cite{AK} improved this to two-point concentration. 
Achlioptas and Naor~\cite{AN} determined the \emph{explicit} two values of $\chi(G(n,d/n))$ for constant $d$, and Coja-Oghlan, Panagiotou and Steger~\cite{CPS} later extended this result to three explicit values for all $p<n^{-3/4-\varepsilon}$. 

Given these results in the sparse case, it is natural to wonder whether $\chi(G(n,p))$ is also concentrated on a bounded number of values when $p$ is constant. In a stunning recent paper,  Heckel~\cite{H2} showed that the answer is no. More precisely, for infinitely many values of $n$ there is no interval of length $n^{1/4-\varepsilon}$ that contains $\chi(G(n,1/2))$ with high probability. This result has since been refined by Heckel and Riordan~\cite{HR}, who improved the lower bound on the concentration to $n^{1/2} (\log n)^{-3 + o(1)}$, within a poly-logarithmic factor of Alon's upper bound, again for some (unknown) infinite sequence of values of $n$. 


Heckel's brilliant idea is based on the influence of the number of independent sets of maximum size on the optimal colourings. A related heuristic suggests that the number of independent sets of size one less than the maximum should also have a significant effect on $\chi(G(n,1/2))$, and this led Bollob{\'a}s, Heckel, Morris, Panagiotou, Riordan and Smith to propose the so-called Zigzag Conjecture (see~\cite[Conjecture 8]{HR}). This conjecture implies that the concentration of $\chi(G(n,1/2))$ depends significantly on $n$, varying between $n^{1/4 + o(1)}$ and $n^{1/2 + o(1)}$, depending on the expected number of independent sets of maximum size. 

While the Zigzag Conjecture remains wide open, Surya and Warnke~\cite{SW} discovered that $\chi(G(n,p))$ exhibits a somewhat similar behaviour in the range $p > 1 - n^{-c}$, for some constant $c > 0$. More precisely, they showed that if $q = 1 - p$ satisfies $n^{-2} \ll q = O(n^{-1})$, then the concentration of $\chi(G(n,p))$ is of order $\Theta(q^{1/2}n)$. They also observed that if 
$$q \gg n^{-2/r} (\log n)^{1/{r \choose 2}},$$ 
then, since $G(n,q)$ has a $K_r$-factor with high probability, by the celebrated theorem of Johansson, Kahn and Vu~\cite{JKV}, the chromatic number of $G(n,p)$ is concentrated on an interval of size $O(\mu_{r+1})$, where $\mu_s = q^{s \choose 2} {n \choose s}$. In particular, this implies that the concentration of $\chi(G(n,p))$ decreases from $n^{1/2}$ to $(\log n)^3$ as $q$ increases from $n^{-1}$ to $n^{-1} \log n$.

Surya and Warnke also made the following beautiful conjecture (see~\cite[Conjecture~8]{SW}) regarding the concentration of the chromatic number when $q$ is not quite so small. If 
$$n^{-2/(r-1)} (\log n)^{1/{r-1 \choose 2}} \ll q \ll n^{-2/r}$$
for some $r \ge 2$, then $\chi(G(n,p))$ is concentrated on an interval of size $O(\sqrt{\mu_r})$, and is not concentrated on any interval of size $o(\sqrt{\mu_r})$. Their result on the concentration of $\chi(G(n,p))$ in the range $n^{-2} \ll q = O(n^{-1})$ confirmed this conjecture in the case $r = 2$. 




In this paper we will study the Surya--Warnke conjecture in the case $r = 3$, that is, when $n^{-1}\log n \ll 1 -p \ll n^{-2/3}$. The following theorem confirms one half of the conjecture in this range, by showing that $\chi(G(n,p))$ is concentrated on an interval of length $O(\sqrt{\mu_3})$. 

\begin{theorem}\label{main1}
For each $\varepsilon > 0$, there exists a constant $C = C(\varepsilon)$ such that the following holds. If\/ $n^{-1}\log n \ll q = 1 -p \ll n^{-2/3}$, then
\begin{align*}
\Prob\Big( \big| \chi\big( G(n,p) \big) -\E\big[ \chi\big( G(n,p) \big) \big] \big| \geq Cn^{3/2}q^{3/2} \Big) \le \varepsilon
\end{align*}
 for all sufficiently large $n \in \N$.
\end{theorem}



We will also provide some further evidence in favour of the Surya--Warnke conjecture by showing that there exists an infinite sequence of values of $n$ such that $\chi(G(n,p))$ is not concentrated on any interval of length $o(\sqrt{\mu_3})$, as long as the function $q = q(n)$ (where $q = 1 - p$) is sufficiently `smooth', which we define as follows. 



\begin{definition}\label{def:smoothfunction}
A function $q \colon \N \to [0,1]$ 
is \emph{smooth} if 
\begin{align*}
q(n+1)=q(n)+o\bigg(\frac{1}{q(n)^2n^3}\bigg)
\end{align*}    
as $n \to \infty$.
\end{definition}

In particular, note that the function $q(n) = n^{-c}$ is smooth for every $c > 2/3$. Inspired by the coupling technique of Heckel~\cite{H2} and Heckel and Riordan~\cite{HR}, we will prove the following non-concentration result for smooth functions. 



\begin{theorem}\label{main2}
If\/ $n^{-1}\log n \ll q = 1 - p \ll n^{-2/3}$ is a smooth function, then for some absolute constants $\varepsilon > 0$ and $c > 0$, and infinitely many values of $n \in \N$, the following holds:
\begin{align*}
\Prob\Big( \chi\big( G(n,p) \big) \in I_n \Big) \le 1 -
 \varepsilon
\end{align*}
for every interval $I_n$ of length $|I_n|\leq cn^{3/2}q^{3/2}$.  
\end{theorem}

Moreover, for a smaller range of $p$ we will be able to prove a central limit theorem for $\chi(G(n,p))$, without any assumption on the smoothness of $q(n)$.  



\begin{theorem}\label{chi:CLT}
If\/ $n^{-1}\log n\ll 1-p\ll n^{-7/9}$, then 
\begin{align*}
\frac{\chi\big( G(n,p) \big) - \E\big[ \chi\big( G(n,p) \big) \big]}{\sqrt{\Var\big( \chi\big( G(n,p) \big) \big)}}\stackrel{d}{\longrightarrow}\mathcal{N}(0,1),
\end{align*}
where $\mathcal{N}(0,1)$ is the standard normal distribution. Moreover, 
$$\Var\big( \chi(G(n,p)) \big) = \Theta\big(n^3(1-p)^3\big).$$ 
\end{theorem}


The main step in the proofs of Theorems~\ref{main1},~\ref{main2} and~\ref{chi:CLT} is the following description of the optimal colouring: with high probability, it is formed by the largest triangle-matching in the graph of non-edges of $G(n,p)$, together with a matching missing at most one vertex. In order to state this result precisely, we need a little notation. 



First, we say that a matching $M$ in a graph $G$ is \emph{near-perfect} if at most one vertex of $G$ is not contained in an edge of $M$. 
A triangle-matching is a collection of vertex-disjoint triangles, and we define $S(G)$ to be the set of vertices of the largest triangle-matching\footnote{If there is more than one triangle-matching of maximum size, then we choose one of them according to some arbitrary deterministic rule, so that $S(G)$ is well-defined.} in $G$. 
Given a set of vertices $S \subset V(G)$, we write $G - S$ for the subgraph of $G$ induced by the set $V(G) \setminus S$. Our structure theorem for optimal colourings is as follows. 


\begin{theorem}\label{structure}
Let $G \sim G(n,q)$ for some $n^{-1}\log n\ll q = q(n) \ll n^{-2/3}$, and set $S = S(G)$. With high probability the graph $G - S$ has a near-perfect matching. 
\end{theorem}

Note that \Cref{structure} determines the optimal
colouring of $G(n,p)$ with high probability, where $p = 1 - q$, since for $q \ll n^{-2/3}$ the expected number of copies of $K_4$ in $G(n,q)$ is $o(1)$. In particular, it follows that, with high probability,
\begin{align}\label{chi}
\chi(G(n,p)) = \bigg\lceil \frac{3n -|S(G(n,q))|}{6} \bigg\rceil.
\end{align}
In order to deduce our three main theorems 
from Theorem~\ref{structure}, it will therefore suffice to control the concentration of the size of the largest triangle-matching in $G(n,q)$.






The rest of this paper is organised as follows. In \Cref{section2} we state our main results about the concentration of $|S(G(n,q))|$, and outline their proofs, and that of \Cref{structure}. 
In \Cref{section3}
we prove a technical result, \Cref{lem:DT:unlikely}, which plays a crucial role in the proof of \Cref{structure}; we then complete the proof of \Cref{structure} in \Cref{section4}. Finally, in Sections~\ref{section*}, Sections~\ref{section5} and~\ref{section6}, we prove our central limit theorem and our bounds on the concentration of $|S(G(n,q))|$. Some discussion of open questions is presented in \Cref{section7}.


\section{Proof overview}\label{section2}

This section has two main aims. First, we will outline the proof of our key structure theorem (\Cref{structure}), which is the most important and novel part of the paper. We will then discuss how we control the concentration of the size of the largest triangle-matching in $G(n,q)$, using martingale techniques (in particular, Freedman's inequality), the classical central limit theorem of Ruci\'nski~\cite{R} for the number of triangles in $G(n,p)$, and the coupling technique introduced by Heckel~\cite{H2} in her recent breakthrough work on the non-concentration of the chromatic number of $G(n,1/2)$. 



\subsection{Finding a near-perfect matching} 



Let $G \sim G(n,q)$ and set $S = S(G)$. To find a near-perfect matching in $G-S$, it's natural to use the celebrated theorem of Hall~\cite{Hall}. Given any set $X$, let us say that $X = A \cup B$ is an \emph{equi-partition} of $X$ if
\begin{align*}
A\cap B = \emptyset \qquad \textup{and} \qquad |A| \leq |B| \leq |A| + 1. 
\end{align*}
Given a graph $H$, we write $N_H(T)$ for the set of vertices of $V(H) \setminus T$ with a neighbour in $T$. 

\begin{lemma}[Hall's theorem]\label{Hall} 
Suppose that $H$ is a bipartite graph whose parts $A$ and $B$ form an equi-partition of $V(H)$. Then there exists a near-perfect matching in $H$ between $A$ and $B$ if and only if $|N_H(T)|\geq |T|$ for every $T\subset A$.
\end{lemma}

To apply Hall's theorem to $G-S$, we choose a random equi-partition $V(G-S) = A \cup B$ of the vertex set. It will then suffice to show that, with high probability (according to the randomness of both $G(n,q)$ and the random partition),
\begin{align}\label{eq:N_{G-S}(T)1}
 |N_{G-S}(T)\cap B| \geq |T|  
\end{align} 
holds for every $T \subset A$. We will prove this in two steps. In the first step, we will use the randomness of $G(n,q)$ to show that with high probability
\begin{align}\label{eq:N_{G-S}(T)2}
|N_{G-S}(T)|\geq(1-\delta)|N_G(T)|   
\end{align}
holds for every $T\subset V(G-S)$, for some small constant $\delta > 0$. We will then, in the second step, use the randomness of the partition $A \cup B$ to show that~\eqref{eq:N_{G-S}(T)1} holds with high probability for any graph $G$ that satisfies~\eqref{eq:N_{G-S}(T)2} and some simple pseudorandom properties.

To see why~\eqref{eq:N_{G-S}(T)2} should hold, note that $\E[|S|] \le n^3q^3 \ll n$, and it is therefore reasonable to hope that only a small fraction of $N_G(T)$ is contained in $S$. If $|T| \ge 1/q$ then this is simple to show, since with high probability we have $|N_G(T)| = \Theta(n) \gg |S|$ for all such sets, so our main challenge will be to show that it holds when $|T| \le 1/q$. 

To do so, we will bound the number of vertices of $N_G(T)$ that are contained in some triangle of $G$. This is in turn at most $3$ times the number of triangles of $G$ that intersect $N_G(T)$. We will use the so-called deletion method of Janson, R\"odl and Ruciński~\cite{JR,RR} to bound the probability that the number of such triangles is much larger than its expectation, allowing us to deduce~\eqref{eq:N_{G-S}(T)2} via a union bound over sets $T$. One minor additional complication is that we will need to deal separately with those triangles than intersect $T$; we will do so using some simple pseudorandom properties of $G$. 

For the second step, we first reveal the graph $G \sim G(n,q)$, and assume that it satisfies~\eqref{eq:N_{G-S}(T)2}, and that $N_G(T)$ has roughly the expected size for every set $T \subset V(G)$ with $|T| = O(1/q)$. We then use the randomness of the partition $V(G - S) = A \cup B$ to bound, for each set $T \subset V(G- S)$, the probability that $T \subset A$ and $| N_{G-S}(T) \cap B | < |T|$. To do so, we split into three cases, depending on the size of $T$. 

First, if $T$ is `small' (more precisely, if $|T| = O(1/q)$), then $|N_G(T)| \gg |T| \log n$ (by our assumption on $G$, and because $q \gg n^{-1} \log n$), which implies that $| N_{G-S}(T) \cap B | < |T|$ holds with probability much smaller than $n^{-|T|}$. Next for `medium-sized' sets $T$ (with $|T| \gg 1/q$ and $n/2 - |T| \gg 1/q$), the set of non-neighbours of $T$ is sufficiently small that~\eqref{eq:N_{G-S}(T)1} holds deterministically. Finally, when $T$ is `huge' (that is, when $n/2 - |T| = O(1/q)$), we need to be careful, since we can only afford to miss a small set of vertices of $B$. Fortunately, however, we can deduce the bound we need by considering the sets $A \setminus T$ and $B \setminus N_G(T)$. 

We will prove in~\Cref{section3} that~\eqref{eq:N_{G-S}(T)2} holds with high probability, and then, in~\Cref{section4}, use our random equi-partition to complete the proof of Theorem~\ref{structure}.

\subsection{Concentration of triangle-matchings in $G(n,q)$}\label{subsection2.1} 

For convenience, let us write $s(H)$ for the number of triangles in the largest triangle-matching in graph $H$, that is, 
$$s(H) = \frac{|S(H)|}{3}.$$
As we noted in the introduction, it follows from the structure theorem (\Cref{structure}) that 
\begin{align}\label{chi}
\chi(G(n,p)) = s(G) + \bigg\lceil \frac{n - 3s(G)}{2} \bigg\rceil = \bigg\lceil \frac{n - s(G)}{2} \bigg\rceil
\end{align}
with high probability, where $G \sim G(n,q)$ is the complement of $G(n,p)$. In order to prove Theorems~\ref{main1},~\ref{main2} and~\ref{chi:CLT}, it will therefore suffice to prove corresponding results about the concentration of $s(G(n,q))$. To be precise, we will prove the following three theorems. The first easily implies Theorem~\ref{main1}, using~\eqref{chi}.




\begin{theorem}\label{main1'}
Let $n^{-1}\log n \ll q \ll n^{-2/3}$. For any $\varepsilon > 0$, there exists $c = c(\varepsilon)$ such that
\begin{align*}
\Prob\Big( \big| s(G(n,q)) -\E[s(G(n,q))] \big| \geq cn^{3/2}q^{3/2} \Big) \le \varepsilon
\end{align*}
for all sufficiently large $n \in \N$.
\end{theorem}

Our second theorem provides a bound on the non-concentration $s(G(n,q))$ for smooth functions $q = q(n)$ (see Definition~\ref{def:smoothfunction}) and infinitely many values of $n \in \N$, and implies Theorem~\ref{main2}, using~\eqref{chi}. The proof is strongly inspired by the method of Heckel~\cite{H2}


\begin{theorem}\label{main2'}
Let $n^{-1}\log n \ll q \ll n^{-2/3}$ be a smooth function. There exists $\varepsilon > 0$ such that, for infinitely many $n \in \N$, we have
\begin{align*}
\Prob\big( s(G(n,q) \in I_n \big) < 1 - \varepsilon
\end{align*}
for every interval $I_n$ of length $|I_n| \leq \varepsilon^2 n^{3/2} q^{3/2}$.  
\end{theorem}

The final property of $s(G(n,q))$ we will need is given by the following theorem. It is a relatively straightforward consequence of the  central limit theorem for the number of triangles in $G(n,q)$, which was proved by Ruci\'nski~\cite{R} in the 1980s. Recall that we write $\mathcal{N}(0,1)$ for the standard normal distribution with mean $0$ and variance $1$. 



\begin{theorem}\label{thm:main*}
Let $n^{-1} \ll q \ll n^{-7/9}$, then
\begin{align*}
\frac{s\big( G(n,q) \big) - \E\big[ s\big( G(n,q) \big) \big]}{\sqrt{\Var\big( s\big( G(n,q) \big) \big)}}\stackrel{d}{\longrightarrow}\mathcal{N}(0,1).
\end{align*}
Moreover, $\Var\big( s(G(n,q)) \big) = \Theta(n^3q^3)$.
\end{theorem}

We remark that a central limit theorem is known for the size of the largest matching in a sparse random graph, see~\cite{Pittel,Kreacic,GKSS}.

We will deduce Theorem~\ref{thm:main*} from Ruci\'nski's theorem in Section \ref{section*}. The reason we can prove the theorem only for $q \ll n^{-7/9}$ is that in this range the expected number of pairs of intersecting triangles in $G(n,q)$ is smaller than the typical deviation in the total number of triangles. It seems plausible that a more complicated version of this argument could allow one to extend the result to a slightly larger range of $q$; however, proving a central limit theorem for the entire range $q \ll n^{-2/3}$ appears to require additional ideas.

In Section \ref{section5}, we will prove Theorem~\ref{main1'} using a martingale method. 
More precisely, we will control the deviations of $s(G(n,q))$ by considering the vertex exposure martingale $X_i$, and apply Freedman's inequality to bound the tail. The main challenge will be to bound the predictable quadratic variation of the martingale; to do so, we will use a simple pseudorandomness condition to bound the probability that a triangle is created in a given step, and then show (via a somewhat technical calculation) that 
$$|X_i - X_{i-1}| = O(n^2q^3)$$ 
(deterministically) whenever a triangle is not created (see \Cref{diff2}). 


Finally, in \Cref{section6}, we will prove Theorem~\ref{main2'} using the method of Heckel~\cite{H2} and Heckel and Riordan~\cite{HR}. The main complication we face in this section is that the method of~\cite{H2,HR} only works when $q$ is constant, whereas our function $q = q(n)$ decreases quite rapidly. To deal with this issue, we construct two couplings, between $G(n,q(n))$ and $G(n',q(n))$, for some $n'$ with $q(n') \sim q(n)$, 
using the method of~\cite{H2,HR}, and then between $G(n',q(n))$ and $G(n',q(n'))$, using sprinkling and the smoothness of $q(n)$. 

Using these two couplings, we will show that if $s(n,q(n))$ 
is concentrated on an interval of size $o\big( n^{3/2}q(n)^{3/2} \big)$ for every sufficiently large $n$, then the typical value of $s(n,q(n))$ must grow linearly with $n$. This will give us our desired contradiction, since the size of the largest triangle-matching in a graph is bounded from above by the total number of triangles. 


\section{Bounding the number of neighbours in $S(G(n,q))$}\label{section3}

In this section we will show that with high probability, only a small fraction of the set $N(T)$ is contained in $S(G(n,q))$ for every $T \subset V(G(n,q))$ with $|T| \le n/2$. For convenience, let's fix for the next two sections a function $n^{-1}\log n\ll q\ll n^{-2/3}$, a sufficiently large integer $n \in \N$, and write $G \sim G(n,q)$ and $S = S(G)$. We also fix a small constant $\delta > 0$.

Recall that $N(x)$ denotes the set of neighbours of $x$ in $G$ and that $N(T) = \bigcup_{x\in T} N(x)\setminus T$. Let us write $\cT$ for the collection of sets $T\subset V(G)$ with $|T| \le n/2$, and for each $T \in \cT$, define
\begin{equation}\label{D(T)}
\cD(T) = \big\{ |N(T)\cap S| > \delta|N(T)| \big\}
\end{equation}
to be the event that more than a $\delta$-proportion of the elements of $N(T)$ are contained in $S$. We will prove that with high probability none of the events $\cD(T)$ hold.

\begin{lemma}\label{lem:DT:unlikely}
\begin{align*}
\Prob\bigg(\bigcup_{T \in \cT} \cD(T) \bigg)=n^{-\omega(1)}.
\end{align*}
\end{lemma}

Proving Lemma~\ref{lem:DT:unlikely} is the main aim of this section. We will bound $\Prob(\cD(T))$ in different ways, depending on the size of $T$. It will be easy to deal with large sets $T$, since when $|T| \geq 1/q$ we will see that with high probability $|N(T)| = \Theta(n)$, whereas $|S| = o(n)$, which implies that $\cD(T)$ fails to hold. 

For small sets $T$, with $|T| \leq 1/q$, things are more complicated. Therefore, instead of bounding the size of $N(T) \cap S$ directly, we will bound a larger (but simpler) object: the number of vertices of $N(T)$ that are contained in \emph{any} triangle of $G$. More precisely, observe that 
$$N(T)\cap S \subset \Lambda_1(T) \cup \Lambda_2(T),$$
where $\Lambda_1(T)$ is the set of vertices of $N(T)$ that are contained in some triangle of $G$ that intersects the set $T$, and $\Lambda_2(T)$ is the set of vertices of $N(T)$ that are contained in some triangle that doesn't intersect $T$. Since every vertex of $S$ is contained in some triangle, by definition, it will suffice to show that
$$|\Lambda_1(T)| + |\Lambda_2(T)| \le \delta |N(T)|.$$
Bounding the size of $\Lambda_1(T)$ will be quite easy: we will just need to estimate the number of triangles that intersect $T$. Bounding $|\Lambda_2(T)|$, on the other hand, will be significantly more difficult; in particular, we will need to use the deletion method of Janson, R\"odl and Ruciński~\cite{JR,RR}. Putting the pieces together, and using the union bound, we will obtain the claimed bound on the probability that $\cD(T)$ holds for some small set $T$. 

Throughout the proof of Lemma~\ref{lem:DT:unlikely}, we will find it convenient to be able to assume that $G$ has certain pseudorandom properties. We will therefore begin by defining the properties we need. Given a set $X \subset V(G)$, let us write $e(X) = e(G[X])$ for the number of edges in the subgraph of $G$ induced by $X$. Also, recalling that $n^{-1}\log n\ll q \ll n^{-2/3}$, let $\omega_0$ be any function tending to infinity sufficiently slowly so that 
\begin{align}\label{omega0}
nq \geq \omega^3_0 \log n \qquad \textup{and} \qquad n^2q^3 \le e^{-\omega_0}.
\end{align}  
We will show that the following event holds with high probability.

\begin{definition}
Let $\cR$ be the event that $G$ has the following four properties:
\begin{itemize}
\item[$(i)$] Let $X_3$ be the number of triangles in $G$, then
$$X_3\leq n^3q^3.$$
\item[$(ii)$] For all $x,y\in V(G)$, 
$$e(N(x))\leq\omega_0\log n \qquad \textup{and} \qquad |N(x)\cap N(y)| \leq \omega_0.$$
\item[$(iii)$] For each $T\subset V(G)$ with $|T|\leq1/q$, 
$$\frac{nq|T|}{2} \leq |N(T)| \leq \frac{3nq|T|}{2}.$$
\item[$(iv)$] For every $T\subset V(G)$ with $1 \le q|T| \le \omega_0$,
$$|N(T)| \geq \big( 1 - e^{-q|T|/2} \big)n.$$
\end{itemize}
\end{definition}
\noindent Note the number of triangles in $S$ is at most $X_3$ by definition, (i) implies $G$ also holds
$$|S|\leq 3X_3\leq 3n^3q^3.$$
It is straightforward to show, using Chernoff's inequality with our selection ~\eqref{omega0}, that the event $\cR$ holds with high probability. We record this fact as the following lemma; for completeness, we provide a proof in the Appendix.

\begin{lemma}\label{regular}
\begin{align*}
\Prob(\cR)=1-n^{-\omega(1)}.
\end{align*}
\end{lemma} 
\smallskip

We are now ready to begin our detailed analysis of the event $\cD(T)$ for sets $T\subset V(G)$ of different sizes. We begin with the simplest case, when $T$ is `large', that is, when $|T|\geq 1/q$. In this case, the event $\cR$ implies that the set $N(T)$ is much larger than $S$.

\medskip
\pagebreak

\begin{observation}\label{ob1}
If the event $\cR$ holds, then
\begin{equation}\label{eq:obs1}
|N(T)\cap S|\leq\delta|N(T)|
\end{equation}
for each set $T\in \cT$ with $|T|\geq 1/q$. 
\end{observation}

\begin{proof}
Since $q|T| \geq 1$ and $|T| \le n/2$, property (\romannumeral4) of $\cR$ implies that $|N(T)| = \Omega(n)$. 
On the other hand, since $q \ll n^{-2/3}$, property (\romannumeral1) of $\cR$ implies that $|S| \leq 3n^3q^3 = o(n)$. It follows that
\begin{align*}
|N(T)\cap S| \leq |S| \ll |N(T)|,
\end{align*}
and therefore~\eqref{eq:obs1} holds, since we chose $n$ sufficiently large.
\end{proof}

When $T \subset V(G)$ is `small', that is, when $|T|\leq 1/q$, the neighbourhood $N(T)$ may no longer be significantly larger than the set $S$, and we will therefore need to work much harder. The first step is to define two sets $\Lambda_1(T)$ and $\Lambda_2(T)$ whose union contains $N(T) \cap S$, and whose sizes we will be able to control. To do so, let us write 
\begin{align*}
\cK_3(G) = \big\{ (x,y,z) : x,y,z\in V(G),\ xy,xz,yz\in E(G) \big\},
\end{align*}
for the collection of triangles of $G$. Now, define
\begin{align*}
\Lambda_1(T) = \big\{ x \in N(T) : \exists \ y\in T, z \in V(G)\ \text{ s.t. }\ (x,y,z)\in \cK_3(G) \big\},
\end{align*}
to be the set of vertices of $N(T)$ that are contained in some triangle of $G$ that intersects the set $T$. Once again, it is straightforward to show that the event $\cR$ implies (deterministically) that the set $\Lambda_1(T)$ is not too large. %


\begin{observation}\label{ob2}
If the event $\cR$ holds, then 
\begin{align*}
|\Lambda_1(T)|<\frac{\delta|N(T)|}{2}
\end{align*}
for each set $T\subset V(G)$ with $|T|\leq 1/q$, 
\end{observation}

\begin{proof}
Note first that, by~\eqref{omega0} and property (\romannumeral3) of $\cR$, we have  
\begin{align*}
|N(T)|\geq\frac{nq|T|}{2} \geq \frac{\omega_0^3|T|\log n}{2}.
\end{align*}
Since $\omega_0 \gg 1$, it will therefore suffice to show that
\begin{equation}\label{eq:lambda1:bound}
|\Lambda_1(T)| \leq 3\omega_0|T|\log n.
\end{equation}
To prove~\eqref{eq:lambda1:bound}, note first that each vertex $y\in T$ is contained in $e(N(y)) \le \omega_0\log n$ triangles of $G$, where the inequality follows from property (\romannumeral2) of $\cR$. Summing over vertices in $T$, and noting that each triangle contains at most three vertices of $N(T)$, we obtain~\eqref{eq:lambda1:bound}.
\end{proof}


We now arrive at the most challenging part of the proof of Lemma~\ref{lem:DT:unlikely}, bounding the size of the set  
\begin{align*}
\Lambda_2(T) = \big\{ x \in N(T) : \exists \  y,z \notin T \ \text{ s.t. }\ (x,y,z)\in \cK_3(G)\big\}
\end{align*}
of vertices of $N(T)$ that are contained in some triangle of $G$ which doesn't intersect $T$. We will do so using the deletion method of Janson, R\"odl and Ruciński~\cite{JR,RR}, so let us first recall the key lemma of that method. Let $k\in\mathbb{N}$, let $\cS$ be a family of subgraphs of $K_n$, each of which has exactly $k$ edges, and define two random variables as follows:
\begin{equation}\label{def:X}
X = \big|\big\{W\in\cS : W \subset G \big\} \big|,
\end{equation}
the number of members of $\cS$ that are contained in $G$, and
\begin{equation}\label{def:Y}
Y^* = \max_{e \in E(K_n)} \big| \big\{ W \in \cS : e \in E(W), \ W \subset G \big\} \big|,
\end{equation}
the maximum over edges $e \in E(K_n)$ of the number of members of $\cS$ contained in $G$ and containing $e$. 
The following lemma can be found in~\cite[Theorem~6A]{JR}. 


\begin{lemma}[The deletion method]\label{deletion}
For each $r,t>0$, let $\cE(r,t)$ be the event that for every graph $F$ with $r$ edges, the graph with edge set $E(G)\setminus E(F)$ contains
at least $\E[X]+t/2$ members of the family $\cS$. Then 
\begin{align*}
\big\{ X \geq \E[X] + t \big\} \subset \cE(r,t) \cup \big\{ Y^* > t/2r \big\}
\end{align*}
and
\begin{equation}\label{eq:Ert:deletion}
\Prob\big( \cE(r,t) \big) \leq \exp\bigg(-\frac{rt}{k\big( 2\E[X]+t \big)}\bigg),
\end{equation}
where $X$ and $Y^*$ are the random variables defined in~\eqref{def:X} and~\eqref{def:Y}.
\end{lemma}

We will use the deletion method to prove the following key lemma. 

\begin{lemma}\label{ob3}
For each $T\subset V(G)$ with $|T|\leq1/q$, 
\begin{align}\label{eq:Lambda2}
\Prob\left(\left\{|\Lambda_2(T)| \ge \frac{\delta|N(T)|}{2}\right\}\cap\cR\right)\leq n^{-\omega_0|T|}.    
\end{align}
\end{lemma}

Since the proof of Lemma~\ref{ob3} is somewhat complicated, let us first give a brief outline of the proof strategy. Define  
\begin{align*}
\hat{\cK}_3(T) = \big\{(x,y,z)\in \cK_3(G) : x\in N(T),\ y,z\notin T\big\}, 
\end{align*}
the collection of triangles of $G$ which intersect $N(T)$ but not $T$, and observe that
\begin{equation}\label{eq:L2:vs:K3}
|\Lambda_2(T)| \le 3 \cdot |\hat{\cK}_3(T)|
\end{equation}
It will therefore sufffice to bound the number of triangles in $\hat{\cK}_3(T)$. 

To do so, we will use the deletion method to bound the size of the set 
\begin{equation}\label{def:ZTA}
Z(T,A) = \big\{ (x,y,z)\in \cK_3(G) : x \in A,\ y,z \notin T \big\}   
\end{equation}
for each set $A \subset V(G) \setminus T$. Note that if $N(T) = A$, then
\begin{equation}\label{eq:ZTA:vs:K3}
|Z(T,A)| = |\hat{\cK_3}(T)|.
\end{equation}

\pagebreak

\noindent Observe that the events 
$$N(T) = A \qquad \text{and} \qquad |Z(T,A)| \ge a$$
are independent (for any $a \ge 0$), since the first depends only on edges incident to $T$, and the second depends only on edges not incident to $T$. It will therefore suffice to show that  
\begin{equation}\label{eq:ZTA:need}
\Pr\Big( \big\{ |Z(T,A)| = \Omega(|A|) \big\} \cap \cR \Big) \le n^{-\omega_0|T|}
\end{equation}
for every $A \subset V(G) \setminus T$. 

To prove~\eqref{eq:ZTA:need}, we apply the deletion method 
to the random variable $X = |Z(T,A)|$, which corresponds to the family $\cS$ consisting of triangles in $K_n$ that intersect $A$ but not $T$. Note that since $q\ll n^{-2/3}$, we have
\begin{align}\label{E[|Z(T,A)|]}
\E[X] \le q^3 |A| n^2 \ll |A|,
\end{align}
so it will suffice to bound the probability that $X \ge \E[X] + t$ for some $t = \Omega(|A|)$. 

By Lemma~\ref{deletion}, to do so we need to choose $r \in \N$ such that the right-hand side of~\eqref{eq:Ert:deletion} is small, and to bound the probability of the event $\big\{ Y^* > t/2r \big\}$. 
Since $t \ge \E[X]$ and $k = 3$, the first condition is satisfied with $r \approx \omega_0|T|\log n$. To bound the probability that $Y^*$ is larger than $t/2r$, we will use the properties guaranteed by the event $\cR$. 



We are now ready to state the main step in the proof of Lemma~\ref{ob3}.

\begin{lemma}\label{lem:ZTA:and:R}
Fix\/ $T\subset V(G)$ with $|T|\leq 1/q$ and $A\subset V(G)\setminus T$ with 
\begin{align}\label{eq:Abounds}
\frac{nq|T|}{2}\leq|A|\leq\frac{3nq|T|}{2}.    
\end{align}
Then for every constant $\varepsilon>0$ we have
$$\Prob\Big(\big\{|Z(T,A)| \ge \varepsilon |A| \big\}\cap\cR\Big) \le n^{-\omega_0|T|}.$$
\end{lemma}

\begin{proof}
We apply the deletion method, with $k = 3$ and 
\begin{align*}
\cS= \big\{ (x,y,z) : x\in A,\ y,z\in V(G) \setminus T \big\},
\end{align*}
as described above. Note that $X = |Z(T,A)|$ and $\E[X] \ll |A|$, by~\eqref{E[|Z(T,A)|]}. Set
\begin{equation}\label{def:t:and:r}
t = \frac{\varepsilon nq|T|}{4} \qquad \text{and} \qquad r = 4\omega_0|T|\log n.
\end{equation}
By Lemma~\ref{deletion} we have
\begin{align}\label{ER0}
\Prob\Big(\big\{ X \geq \E[X] + t \big\} \cap \cR \Big) \leq \Prob\big( \cE(r,t) \big) + \Prob\Big( \big\{Y^*\geq t/2r \big\}\cap\cR \Big) 
\end{align}
and 
\begin{align}\label{eq:Ert:app}
\Prob\big( \cE(r,t) \big) \leq \exp\left(-\frac{rt}{3(2\E[X]+t)}\right) \le e^{-r/4} = n^{-\omega_0|T|},
\end{align}
since $\E[X] \ll t = \Theta(|A|)$. 

To prove the lemma, it will therefore suffice to prove the following claim. 



\medskip
\pagebreak

\begin{claim}\label{claim:Ystar}
If the event $\cR$ holds, then 
$$Y^* \le \omega_0 < t/2r.$$
\end{claim}

\begin{proof}[Proof of claim]
The second inequality follows from~\eqref{omega0} and the definition~\eqref{def:t:and:r} of $t$ and~$r$. Indeed, we have 
$$\frac{t}{2r} = \frac{\varepsilon n q}{32\omega_0 \log n} \ge \frac{\varepsilon\omega_0^2}{32} > \omega_0,$$
since $nq \geq \omega^3_0 \log n$. For the first inequality, recall that, by property (\romannumeral2) of $\cR$, the size of common neighbourhood of any pair of vertices in $G$ is at most $\omega_0$. Since $Y^*$ is at most the maximum over pairs of vertices $x$ and $y$ of the number of triangles in $G$ containing $x$ and $y$, it follows that $Y^* \le \omega_0$, as claimed. 
\end{proof}

By~\eqref{ER0},~\eqref{eq:Ert:app} and Claim~\ref{claim:Ystar}, and recalling that $\E[X] \ll |A|$ and that $t \le \varepsilon |A|/2$, it follows that
$$\Prob\Big(\big\{X \ge \varepsilon |A| \big\}\cap\cR\Big) \le n^{-\omega_0|T|},$$
as required.
\end{proof}

We can now easily deduce Lemma~\ref{ob3}

\begin{proof}[Proof of \Cref{ob3}]
Set $\varepsilon = \delta /6$, and suppose that $\cR$ holds and $|\Lambda_2(T)| \ge 3\varepsilon |N(T)|$. Observe that if $N(T) = A$, then 
$$|Z(T,A)| \ge \frac{|\Lambda_2(T)|}{3} \ge \varepsilon |A|,$$
by~\eqref{eq:L2:vs:K3} and~\eqref{eq:ZTA:vs:K3}, and recall from~\eqref{eq:ZTA:need} that the events 
$$N(T) = A \qquad \text{and} \qquad |Z(T,A)| \ge \varepsilon |A|$$
are independent for any fixed set $A \subset V(G) \setminus T$. Moreover, if $N(T) = A$ then~\eqref{eq:Abounds} holds, 
by property (\romannumeral3) of $\cR$. It follows that 
$$\Prob\Big( \big\{|\Lambda_2(T)| \ge 3\varepsilon|N(T)| \big\} \cap \cR \Big) \leq \sum_{A \in \cI(T)}\Prob\Big(\big\{|Z(T,A)|\geq \varepsilon|A| \big\}\cap \cR \Big) \cdot \Prob\big( N(T) = A \big),$$
where $\cI(T)$ is the family of sets $A \subset V(G) \setminus T$ satisfying~\eqref{eq:Abounds}. Now, since

$$\Prob\Big(\big\{|Z(T,A)| \ge \varepsilon |A| \big\}\cap\cR\Big) \le n^{-\omega_0|T|}$$
for every $A \in \cI(T)$, by Lemma~\ref{lem:ZTA:and:R}, the lemma follows.
\end{proof}

We're now finally ready to prove Lemma~\ref{lem:DT:unlikely}.

\begin{proof}[Proof of Lemma~\ref{lem:DT:unlikely}]
Recall that $\cT$ is the collection of sets $T\subset V(G)$ with $|T| \le n/2$, and that $\cD(T)$ denotes the event that $|N(T)\cap S| > \delta|N(T)|$. Note that, by the union bound, 
\begin{align}\label{D_0}
\Prob\bigg(\bigcup\limits_{T \in \cT}\cD(T)\bigg) \leq \, \Prob\bigg( \bigcup\limits_{|T| \le 1/q} \cD(T)\cap \cR \bigg) + \, \Prob\bigg( \bigcup\limits_{1/q < |T|\le n/2}\cD(T) \cap \cR \bigg) + \, \Prob(\cR^c).
\end{align}
Recall that 
\begin{align}\label{kkkk1}
\Prob(\cR^c) = n^{-\omega(1)},   
\end{align}
by~\Cref{regular}, and that 
\begin{align}\label{kkkk2}
\Prob\big( \cD(T)\cap \cR \big) = 0    
\end{align}
for every set $T$ with $1/q < |T|\le n/2$,  by~\Cref{ob1}. Finally, since
$$N(T)\cap S \subset \Lambda_1(T) \cup \Lambda_2(T),$$
it follows from~\Cref{ob2} and~\Cref{ob3} that 
\begin{align*}
\Prob\big( \cD(T)\cap \cR \big) \leq n^{-\omega_0|T|}
\end{align*}
for every set $T$ with $|T| \le 1/q$. Taking a union bound over such sets, it follows that
\begin{equation}\label{kkkk3}
\Prob\bigg( \bigcup\limits_{|T| \le 1/q} \cD(T)\cap \cR \bigg) \leq \,  
\sum_{\ell \leq 1/q}\binom{n}{\ell}\cdot n^{-\omega_0\ell} \,\leq\, \sum_{\ell\leq 1/q} (n\cdot n^{-\omega_0})^{\ell} = n^{-\omega(1)}.
\end{equation}
Combining~\eqref{D_0}, with~\eqref{kkkk1},~\eqref{kkkk2} and~\eqref{kkkk3}, we deduce that
\begin{align*}
\Prob\bigg(\bigcup_{T \in \cT} \cD(T) \bigg) = n^{-\omega(1)},
\end{align*}
as required.
\end{proof}

\section{Finding a near-perfect matching in $G-S$}\label{section4}

In this section we will find a near-perfect matching in $G - S$, and hence prove \Cref{structure}. We will do so deterministically, for any graph $G$ for which the event $\cR$ holds, and the event in~\Cref{lem:DT:unlikely} does not hold. To be precise, define
\begin{align}
\cG = \Big\{ H \subset K_n : H \in \cR \,\text{ and }\, H \not\in \cD(T) \,\text{ for every }\, T \in \cT \Big\}\label{cggg}
\end{align}   
to be the collection of graphs on $n$ vertices such that the event $\cR$ holds, the event $\cD(T)$ does not hold for any set $T \subset V(G)$ with $|T| \le n/2$. Recall from Lemmas~\ref{lem:DT:unlikely} and~\ref{regular} that the graph $G \sim G(n,q)$ is a member of the family $\cG$ with probability at least $1 - n^{-\omega(1)}$.


The aim of this section is the following deterministic theorem. 

\begin{theorem}\label{structure'}
If\/ $G \in \cG$, then there exists a near-perfect matching in $G - S$.
\end{theorem}

As noted above, Theorem~\ref{structure} follows from Theorem~\ref{structure'}, together with Lemmas~\ref{lem:DT:unlikely} and~\ref{regular}. The key idea in the proof is to consider a \emph{random equipartition} $V(G') = A \cup B$ of the vertex set of $G' = G - S$. In this section, every  probabilistic statement will be with respect to this probability space. We will show, for every graph $G \in \cG$, that with high probability
\begin{align*}
|N_{G'}(T)\cap B| \ge |T|
\end{align*}
holds for every set $T\subset A$. By Hall's theorem (\Cref{Hall}), this implies the existence of an equipartition with a near-perfect matching, as required.


In the proof, we will find it more convenient to work with a uniformly-chosen random subset $A \subset V(G')$, and set $B = V(G')\setminus A$. In other words, each vertex $u \in V(G')$ falls into either $A$ or $B$ with probability $1/2$, independently. Write 
\begin{align}\label{def:F}
\cF = \big\{ |A| \leq |B| \leq |A|+1 \big\}  
\end{align}
for the event that $V(G') = A \cup B$ is an equipartition, and observe that
\begin{align*}
\Prob(\cF) \ge \frac{1}{2\sqrt{n}},
\end{align*}  
by Stirling's formula. Now, for each $T \subset V(G')$, let $\pi(T)$ denote the `bad' event
\begin{align}\label{def:pi}
\pi(T) = \big\{ T \subset A \big\} \cap \big\{ | N_{G'}(T) \cap B | < |T| \big\}.
\end{align}
We will show that
\begin{align}\label{eq:F:and:Hall}
\Prob\Big(\cF \cap \bigcup\limits_{T\subset V(G')}\pi(T)\Big)\ll\frac{1}{\sqrt{n}},
\end{align}
and hence that, conditional on $\cF$ holding, with high probability we have
\begin{align*}
\big|N_{G'}(T)\cap B\big| \ge |T|
\end{align*}
for every $T\subset A$, as required.

In order to prove~\eqref{eq:F:and:Hall}, we will partition into three cases depending on the size of $T$. We begin with the most challenging case, when $T$ is small. Let $C$ be a sufficiently large constant, and define 
\begin{align}\label{def:W1}
\cW_1 = \big\{ T\subset V(G') : |T| \leq C/q \big\},
\end{align}
the collection of `small' subsets of $V(G')$. 

\begin{lemma}\label{p1}
\begin{align*}
\Prob\bigg( \bigcup_{T \in \cW_1} \pi(T) \bigg) = n^{-\omega(1)}.
 \end{align*}   
\end{lemma}

\begin{proof}
Fix a set $T \in \cW_1$, so $T\subset V(G')$ and $1 \leq |T| \leq C/q$. The event $\pi(T)$ implies that there exists a set $X \subset N_{G'}(T)$ with $|X| < |T|$ such that 
$$T\subset A, \qquad X \subset B \qquad \text{and} \qquad N_{G'}(T)\setminus X \subset A.$$
Note that these three events are independent. Writing 
\begin{align*}
\cX(T) = \big\{ X \subset N_{G'}(T) : |X| < |T| \big\}   
\end{align*}
for the collection of subsets of $N_{G'}(T)$ of size at most $|T|$, it follows that
$$\Prob\big(\pi(T)\big) \le \sum_{X\in\cX(T)} \Prob(T \subset A)\cdot\Prob(X \subset B)\cdot\Prob(N_{G'}(T)\setminus X \subset A).$$
Since the set $A \subset V(G')$ was chosen uniformly at random, it follows that 
\begin{align}\label{asd}
\Prob\big(\pi(T)\big)
& \leq \sum_{X\in\cX(T)} 2^{-|T|}\cdot 2^{-|X|}\cdot 2^{-|N_{G'}(T)| + |X|} \,\le\, n^{|T|}\cdot 2^{-|N_{G'}(T)|}.
\end{align}
Now, since $G \in \cG$, we have
\begin{align}\label{asdf}
|N_{G'}(T)| > (1 - \delta)|N_G(T)| \gg |T| \log n,
\end{align}
where the first inequality follows because the event $\cD(T)$ does not hold, and the second follows from property (\romannumeral3) of $\cR$ when $|T| \le 1/q$ (since $q \gg n^{-1} \log n$), and from property (\romannumeral4) of $\cR$ when $1/q \le |T| \le C/q$ (since $q \ll n^{-2/3}$).


Combining~\eqref{asd} and~\eqref{asdf}, and taking a union bound over $T \in \cW_1$, we deduce that
\begin{align*}
\Prob\bigg(\bigcup_{T\in\cW_1}\pi(T)\bigg) 
 & \leq \sum_{\ell = 1}^{C/q} \binom{n}{\ell}\cdot n^{-\omega(\ell)} = n^{-\omega(1)},  
\end{align*}
as claimed.
\end{proof}

We next deal with the `medium-sized' sets, 
\begin{align*}
\cW_2 = \big\{ T \subset V(G') : C/q \leq |T| \leq v(G')/2 - C/q \big\}.    
\end{align*}
We will show that if $A \cup B$ is an equipartition, then none of these bad events can occur.


\begin{lemma}\label{p2}
\begin{align*}
\Prob\bigg( \cF \cap \bigcup_{ T \in \cW_2} \pi(T) \bigg) = 0.
 \end{align*}   
\end{lemma}

\begin{proof}
We claim that, for every $T \in \cW_2$, we have
\begin{equation}\label{eq:p2:claim}
|N_{G'}(T) \cup T| \ge v(G')/2 + |T|,
\end{equation}
and hence if $A \cup B$ is an equipartition then 
$$|N_{G'}(T) \cap B | \ge |N_{G'}(T) \cup T| - |A| \ge |T|,$$ 
as required, since $|A| \le v(G')/2$ and $T \subset A$.

To prove~\eqref{eq:p2:claim}, recall that by property (\romannumeral4) of $\cR$, we have
\begin{equation}\label{eq:R4:recall}
|N_G(U)| \geq \big( 1 - e^{-q|U|/2} \big)n
\end{equation}
for every $U \subset V(G)$ with $1 \le q|U| \le \omega_0$. Let $T \in \cW_2$, and let $U \subset T$ with $|U| = C/q$. By~\eqref{eq:R4:recall}, it follows that
$$|N_{G'}(T)| \ge |N_G(U)| - |S| \geq \big( 1 - e^{-C/2} \big)n - o(n),$$
since $|S| = o(n)$, by property (\romannumeral1) of $\cR$. Therefore, recalling that $C$ and $n$ are sufficiently large,~\eqref{eq:p2:claim} holds unless $|T| > n/3$, say. If $|T| > n/3$, however, then~\eqref{eq:R4:recall} implies that 
$$|N_G(T) \cup T| > n - C/q,$$
since every set of size $C/q$ has a neighbour in $T$. Since $|T| \le v(G')/2 - C/q$ for every $T \in \cW_2$, it follows that 
$$|N_{G'}(T) \cup T| \ge v(G') - C/q \ge v(G')/2 + |T|,
$$
as claimed.
\end{proof}

It remains to deal with the `large' sets $T$, that is, those in the family 
\begin{align*}
\cW_3 = \big\{ T \subset V(G') : v(G')/2 -C/q \leq |T| \leq v(G')/2 \big\}.    
\end{align*}
We will use symmetry to deduce the following bound from Lemma~\ref{p1}. 

\begin{lemma}\label{p3}
\begin{equation}\label{eq:p3}
\Prob\Big(\cF\cap\bigcup_{T\in\cW_3}\pi(T)\Big)=n^{-\omega(1)}.
 \end{equation}   
\end{lemma}

\begin{proof}
Let $T \in \cW_3$, and observe that if $T \subset A$, $|A| \le |B|$ and $|N_{G'}(T)\cap B| < |T|$, then
\begin{align}\label{eq:swapping:to:small}
|B \setminus N_{G'}(T)| > |B| - |T| \ge |A| - |T| = |A \setminus T|.   
\end{align}
Set $U = B \setminus N_{G'}(T)$, and  observe that if $T \subset A$, then $N_{G'}(U) \cap T  = \emptyset$. This implies that 
\begin{align}\label{eq:swapped}
|N_{G'}(U) \cap A| \leq |A \setminus T| < |U|,
\end{align}
by~\eqref{eq:swapping:to:small}. This motivates the following definition: for each $U \subset V(G')$,
\begin{align*}
\pi'(U) = \big\{ U \subset B \big\} \cap \big\{ |N_{G'}(U) \cap A| < |U| \big\}. 
\end{align*}
The following claim will allow us to deduce~\eqref{eq:p3} from Lemmas~\ref{p1} and~\ref{p2}. 


\begin{claim}
$$\cF \cap \bigcup_{T \in \cW_3}\pi(T) \qquad \Rightarrow \qquad \bigcup_{U \in \cW_1 \cup \cW_2} \pi'(U).$$
\end{claim}

\begin{proof}
Fix a set $T \in \cW_3$, and assume that $\pi(T)$ holds, so $T \subset A$ and $|N_{G'}(T)\cap B| < |T|$. By~\eqref{eq:swapped}, it follows that the event $\pi'(U)$ holds, where $U = B \setminus N_{G'}(T)$. It therefore only remains to show that $U \in \cW_1 \cup \cW_2$; that is, that $|U| \le v(G')/2 - C/q$. 

We will in fact show that $|U| = o(n)$. Since $G \in \cG \subset \cR$, and therefore $|S| = o(n)$, this will suffice to prove the claim. To do so, 
we choose $T_0 \subset T$ with $|T_0|=\omega_0/q$, and note that 
$$|N_{G'}(T)| \geq |N_G(T)| - |S| \geq |N_G(T_0)| - |S| = n - o(n),$$
by properties (i) and (iv) of $\cR$, since $\omega_0 \to \infty$ as $n \to \infty$ and $|S| \leq 3n^3q^3 =o(n)$.


Now, since $B \subset V(G')$ and $v(G') \le n$, it follows that
$$|U| = |B \setminus N_{G'}(T)| \le n - |N_{G'}(T)| = o(n),$$
as claimed.
\end{proof}


Now, by the proofs of Lemmas~\ref{p1} and~\ref{p2}, applied to the event $\pi'(U)$ in place of $\pi(T)$, we have
\begin{align*}
\Prob\bigg( \cF \cap\bigcup_{U \in \cW_1 \cup \cW_2} \pi'(U) \bigg) = n^{-\omega(1)}, 
\end{align*}
and therefore~\eqref{eq:p3} follows from the claim.
\end{proof}

We are now ready to prove \Cref{structure'}.

\begin{proof}[Proof of \Cref{structure'}]
Given a graph $G \in \cG$, set $G' = G - S$, let $A$ be a uniformly-chosen random subset of $V(G')$, and set $B = V(G') \setminus S$. Let $\cF$ be the event~\eqref{def:F} that $V(G') = A \cup B$ is an equipartition, and for each $T \subset V(G')$, let the event $\pi(T)$ be as defined in~\eqref{def:pi}. 

Observe that, by Lemmas~\ref{p1}, ~\ref{p2} and~\ref{p3}, we have
\begin{align}\label{2.3}
\Prob\bigg(\cF \cap \bigcup\limits_{T\subset V(G')}\pi(T)\bigg) = n^{-\omega(1)}.
\end{align}
Since $\Prob(\cF) \ge 1/2\sqrt{n}$, it follows that there exists an equipartition $V(G') = A \cup B$ such that 
\begin{align*}
|N_{G'}(T)\cap B| \ge |T|
\end{align*}
for every set $T \subset A$. By Hall's theorem (\Cref{Hall}), this implies the existence of an equi-partition with a near-perfect matching, as required.
\end{proof} 

\section{A Central Limit Theorem for $q \ll n^{-7/9}$}\label{section*}

In this section, we prove a central limit theorem for $s(G(n,q))$ when $n^{-1}\ll q\ll n^{-7/9}$. Our proof is based on the following well-known result of Ruci\'nski~\cite[Theorem 2]{R} on the number of small subgraphs in $G(n,p)$. 

\begin{theorem}[Ruci\'nski]\label{lem:Rucinski}
Let $H$ be an arbitrary graph with at least one edge, and let $X_H$ be the number of copies of $H$ in $G(n,p)$. Then
\begin{align*}
\frac{X_H-\E[X_H]}{\sqrt{\Var(X_H)}}\stackrel{d}{\longrightarrow}\mathcal{N}(0,1)
\end{align*}
if and only if $np^{m}\rightarrow\infty$ and $n^2(1-p)\rightarrow\infty$, where $m=\max\{e(W)/|W|:W\subset H\}$.
\end{theorem}

We will in fact only need the following special case of Theorem~\ref{lem:Rucinski}. Let's write $x(G)$ to denote the number of triangles in a graph $G$. 


\begin{corollary}\label{cor:xnormal}
Let $n^{-1} \ll q \ll 1$, and let $G \sim G(n,q)$. Then 
 \begin{align*}
\frac{x(G) - \E[x(G)]}{\sqrt{\Var(x(G))} } \stackrel{d}{\longrightarrow} \mathcal{N}(0,1)
 \end{align*}
\end{corollary}

\begin{proof}
Apply~\Cref{lem:Rucinski} with $H = K_3$, and note that $e(W) \le |W|$ for every $W \subset K_3$.
\end{proof}



To deduce \Cref{thm:main*} from \Cref{cor:xnormal}, we just need to show that if $n^{-1} \ll q \ll n^{-7/9}$ and $G \sim G(n,q)$, then $s(G)$ is sufficiently close to $x(G)$ with high probability. In order to control the difference $x(G) - s(G)$, we introduce another random variable. Let $y(G)$ be the number of triangles of $G$ that intersect another triangle of $G$, and observe that
\begin{align}\label{eq:s(G)<x(G)<s(G)+y(G)}
s(G) \leq x(G) \leq s(G) + y(G),  
\end{align}
since each triangle in $G$ that is not included in the largest triangle matching must have a common vertex with some other triangle. Hence to control the difference between $x(G)$ and $s(G)$, it will suffice to control $y(G)$. We will do so using the following moment bounds. 

\begin{lemma}\label{lem:y(G)}
If $q \ll n^{-7/9}$, then
\begin{align*}
\E\big[ y(G) \big] = o\big(n^{3/2}q^{3/2} \big) \qquad \textup{and} \qquad \E\big[ y(G)^2 \big] = o\big( n^{3}q^{3} \big).    
\end{align*}   
\end{lemma}

\begin{proof}
To bound the expectation of $y(G)$, simply note that the expected number of pairs of triangles sharing a vertex is at most $n^5q^6$, and the expected number of pairs of triangles sharing an edge is at most $n^4q^5$. It follows that
\begin{equation}\label{eq:ExyG}
\E\big[ y(G) \big] \le n^5q^6 + n^4q^5 \ll n^{3/2}q^{3/2},
\end{equation}
as claimed, since $q \ll n^{-7/9}$. 

To bound the second moment, we need to consider all pairs of triangles $(S,T)$ in $G$, each of which intersects another triangle of $G$. There are various cases to consider. Note first that the expected number of such pairs with $S = T$ is just $\E\big[ y(G) \big]$, and
$$\E\big[ y(G) \big] \le n^5q^6 + n^4q^5 \ll n^3q^3,$$
by~\eqref{eq:ExyG}, and since $q \ll n^{-2/3}$. The expected number of pairs $(S,T)$ such that $S$ and $T$ share a vertex can also be bounded in the same way, so we may assume that $S$ and $T$ are disjoint.  




Consider next the case where there exists a triangle $U$ that intersects both $S$ and $T$. The expected number of such triples is at most 
$$n^6q^8 + n^7q^9 \ll n^3q^3,$$
since $U$ may or may not contain a vertex outside $S \cup T$, and $q \ll n^{-2/3}$. 

Finally, the expected number of quadruples $(S,S',T,T')$ of distinct triangles in $G$ such that $S$ and $T$ are disjoint, $S$ intersects $S'$ but not $T$, and $T$ intersects $T'$ but not $S$, is at most
$$n^8 q^{10} + n^9 q^{11} + n^{10} q^{12} \ll n^3q^3,$$
as required, since $q \ll n^{-7/9}$.
\end{proof}


Combining Lemma~\ref{lem:y(G)} with~\eqref{eq:s(G)<x(G)<s(G)+y(G)}, we can now deduce that $s(G)$ is `close' to $x(G)$.

\medskip
\pagebreak

\begin{lemma}\label{lem:x(G)sims(G)}
Let $n^{-1} \ll q \ll n^{-7/9}$, and let $G \sim G(n,q)$. Then 
\begin{equation}\label{eq:xminuss}
|x(G) - s(G)| = o\big( n^{3/2} q^{3/2} \big)
\end{equation}  
with high probability, and moreover 
\begin{equation}\label{eq:Varx:Vars}
\Var\big( s(G) \big)  = \big(1+o(1)\big) \, \Var\big( x(G) \big)
\end{equation}  
as $n \to \infty$. 
\end{lemma}

\begin{proof}
Let $\delta = \delta(n) \to 0$ as $n \to \infty$ sufficiently slowly. 
By Lemma~\ref{lem:y(G)} and Markov's inequality, we have 
\begin{align*}
\Prob\big( y(G) \ge \delta n^{3/2} q^{3/2} \big) \leq  \frac{\E[y(G)]}{\delta n^{3/2} q^{3/2}} = o(1). 
\end{align*}
By \eqref{eq:s(G)<x(G)<s(G)+y(G)}, this implies~\eqref{eq:xminuss}, since we have
\begin{align*}
|x(G)-s(G)|\leq y(G)=o(n^{3/2}q^{3/2}) 
\end{align*}   
with high probability. To prove~\eqref{eq:Varx:Vars}, observe first that
\begin{equation}\label{eq:Var:Cov}
\Var\big( s(G) \big) = \Var\big( x(G) \big) + \Var\big( x(G) - s(G) \big) + 2 \cdot \textup{Cov}\big( x(G),s(G)-x(G)\big).
\end{equation}  
To bound the right-hand side of~\eqref{eq:Varx:Vars}, recall that $\Var(x(G)) = \Theta(n^3q^3)$, and that
\begin{align}\label{eq:VarCov2}
\Var\big( x(G) - s(G) \big) \le \E\big[ (x(G) - s(G))^2 \big] \leq \E\big[ y(G)^2 \big] = o\big( n^{3} q^{3} \big),
\end{align}
by Lemma~\ref{lem:y(G)}. It follows that
\begin{align}\label{eq:VarCov2}
\textup{Cov}\big( x(G),s(G)-x(G)\big) \le \Big( \Var\big( (x(G) \big) \cdot \Var\big( (s(G) - x(G) \big)  \Big)^{1/2} = o\big( n^{3} q^{3} \big),
\end{align}
and hence 
\begin{equation}\label{eq:Var:Cov}
\Var\big( s(G) \big) = \Var\big( x(G) \big) + o\big( n^{3} q^{3} \big) = \big(1+o(1)\big) \, \Var\big( x(G) \big),
\end{equation}  
as required.
\end{proof}

We are now ready to deduce the Central Limit Theorem for $s(G)$.

\begin{proof}[Proof of \Cref{thm:main*}]
Let $q = q(n)$ be a function  satisfying $n^{-1} \ll q \ll n^{-7/9}$, and let $G \sim G(n,q)$. Let $Z \sim \mathcal{N}(0,1)$, and set 
$$s^*(G) = \frac{s(G)-\E[s(G)]}{\sqrt{\Var(s(G))}} \qquad \text{and} \qquad x^*(G) = \frac{x(G) - \E[x(G)]}{\sqrt{\Var(x(G))}}.$$ 
It suffices to show that for arbitrary constants $a<b$,  
\begin{align}\label{eq:inti}
\Prob\big(s^*(G) \in (a,b)\big) = \Prob\big( Z \in (a,b) \big) + o(1)
\end{align}
as $n\rightarrow\infty$. We will deduce this from \Cref{cor:xnormal} and the following claim. 


\begin{claim}
With high probability, 
\begin{align*}
s^*(G) =\big(1+o(1)\big) \, x^*(G) + o(1).
\end{align*}    
\end{claim}

\begin{proof}
This follows easily from Lemma~\ref{lem:x(G)sims(G)} and the fact that $\Var\big( x(G) \big) = \Theta(n^3q^3)$. 
\end{proof}


Now, by the claim, there exist $a' = a + o(1)$ and $b' = b + o(1)$ such that
\begin{align*}
\Prob\big( s^*(G) \in (a,b) \big) = \Prob\big( x^*(G) \in (a',b') \big) + o(1).    
\end{align*}
Since $x^*(G)$ converges to the standard normal, by \Cref{cor:xnormal}, we have
\begin{align*}
\Prob\big( x^*(G) \in (a',b') \big) = \Prob\big( Z \in (a',b') \big) + o(1),
\end{align*}
and hence, by the continuity of normal distribution, 
\begin{align*}
\Prob\big(s^*(G) \in (a,b) \big) = \Prob\big( Z \in (a',b') \big) + o(1) = \Prob\big( Z \in (a,b) \big) + o(1),   
\end{align*}
as required.
\end{proof}

\section{Concentration: Proof of Theorem \ref{main1'}}\label{section5}

For convenience, let us fix throughout this section a constant $\varepsilon > 0$ (as in the statement of Theorem~\ref{main1'}), and let $n \in \N$ be sufficiently large. Let $G \sim G(n,q)$ for some $q = q(n)$ with $n^{-1}\log n\ll q\ll n^{-2/3}$, and let $V(G) = \{v_1,v_2,...,v_n\}$. Recall that for any graph $H$, we write $s(H)$ for the number of triangles in the largest triangle-matching in $H$, that is, 
$$s(H)=\frac{|S(H)|}{3}.$$
We will use martingales to prove that $s(G)$ is concentrated on an interval of order $n^{3/2}q^{3/2}$. For each $i \in [n]$, 
write $\cG_i$ for the collection of graphs with vertex set $A_i = \{v_1,\ldots,v_i\}$, and for each $i \le j \le n$, and every graph $H \in \cG_j$, define $H_i = H[A_i]$ to be the subgraph of $H$ induced by the first $i$ vertices. Now, define 
\begin{align}\label{Xi}
X_i = \E\big[ s(G) \,|\, G_i \big]
\end{align}
to be the vertex exposure martingale of $s(H)$. Notice that, by definition,
\begin{align*}
X_0 = \E\big[ s(G) \big] \qquad\textup{ and } \qquad X_{n} = s(G).
\end{align*}
To bound the deviation of $s(G)$ we will use the following inequality of Freedman~\cite{F}.

\begin{lemma}[Freedman's inequality]\label{Freedman}
Let $\{X_i\}_{i=1}^n$ be a martingale with respect to a filtration $\{\cF_i\}_{i=1}^n$ and suppose that
\begin{align*}
    |X_i-X_{i-1}|\leq r 
\end{align*}
for some $r>0$. Let $V_n$ be the predictable quadratic variation of $\{X_i\}_{i=1}^n$
\begin{align}\label{VVV}
V_n=\sum_{i=1}^{n}\E\left[(X_i-X_{i-1})^2 \,\big|\, \cF_{i-1}\right].
\end{align}
Then for every $t,\sigma^2>0$,
\begin{align}\label{PPP}
\Prob\big(|X_{n}-X_0|\geq t\big)\leq\exp\left(-\frac{t^2}{2\sigma^2+rt}\right)+\Prob\left(V_n\geq\sigma^2\right).
\end{align}
\end{lemma}

To apply Freedman's inequality, we need two ingredients: a bound on the uniform difference $|X_i - X_{i-1}|$ that holds deterministically, and a bound on the large deviations of the quadratic variation $V_n$ of $\{X_i\}_{i=1}^n$. The uniform difference turns out to be easy to bound.

\begin{observation}\label{diff1}
For each $1 \leq i \leq n$,
\begin{align}\label{uniformbound}
 |X_i - X_{i-1}| \leq 1.   
\end{align}
\end{observation}

\begin{proof}
Note that changing the neighbourhood of $v_i$ cannot change $s(G)$ by more than $1$, since $v_i$ can be used in at most one triangle. The claimed bound~\eqref{uniformbound} follows easily.
\end{proof}


Dealing with the quadratic variation $V_n$ is much more challenging. Most of our energy in this section will be spent on proving the following lemma.

\begin{lemma}\label{lem:V_n}
$$\Prob\big( V_n \ge 4n^3q^3 \big) = n^{-\omega(1)}.$$
\end{lemma}

Assuming \Cref{lem:V_n}, we can directly apply Freedman's inequality to prove \Cref{main1'}.

\begin{proof}[Proof of Theorem \ref{main1'}] 
Our task is to prove that
\begin{align*}
\Prob\Big( | X_n - X_0 | \geq 3 \big( n^3 q^{3} \log(1/\varepsilon) \big)^{1/2} \Big) \le \varepsilon,
\end{align*}
where $\{X_i\}_{i=1}^n$ is the vertex exposure martingale of $s(G)$ defined in \eqref{Xi}. By \Cref{diff1} and \Cref{lem:V_n}, we can apply Freedman's inequality with  
\begin{align*}
r = 1, \qquad  t = 3 \big( n^3 q^{3} \log(1/\varepsilon) \big)^{1/2} \qquad \text{and} \qquad \sigma^2 = 4n^3q^3,
\end{align*}
and deduce that
\begin{align*}
\Prob\big( |X_n - X_0| \geq t \big) \leq \exp\left( - \frac{9 n^3 q^3 \log(1/\varepsilon)}{8n^3q^3 + n^{3/2} q^{3/2} \log(1/\varepsilon)}\right) + n^{-\omega(1)} \leq \varepsilon
\end{align*}
as required, since $n$ is sufficiently large and $nq \gg \log n$.
\end{proof}

It remains to prove \Cref{lem:V_n}. Note first that the bound
\begin{align*}
V_n \leq n    
\end{align*}
holds deterministically, by \eqref{uniformbound}. However, since $q \ll n^{-2/3}$, this bound is weaker than the one we want. To obtain a better bound on $V_n$, we will partition the collection $\cG_i$ of graphs with vertex set $A_i = \{v_1,\ldots,v_i\}$ according to whether or not the vertex $v_i$ is contained in a triangle in $G_i$. If it is, then we will use the trivial bound~\eqref{uniformbound}; if not, then we will show that either $|X_i - X_{i-1}| = O(n^2q^3)$, or some vertex of $G$ has unusually high degree.


To be precise, for each $j \in [n]$, define 
\begin{align}\label{NNN}
\cN_j = \big\{ H \in \cG_j \,:\, | N_{H_i}(v_i)| \leq 3qn \, \text{ for all }\, 1 \le i \leq j \big\}.
\end{align}
In other words, $\cN_j \subset \cG_j$ is the collection of graphs in which every vertex has at most $3qn$ `backwards' neighbours. Note that, since $qn \gg \log n$, the random graph $G_j$ is in $\cN_j$ with very high probability. We record this simple fact as the following observation. 

\begin{observation}\label{obs:degrees:ok}
For each $1 \leq i \leq n$,
\begin{align*}
\Prob\big( G_i \in \cN_i \big) = 1 - n^{-\omega(1)}. 
\end{align*}
\end{observation}







We now partition the family $\cN_i$ according to whether or not $v_i$ is in a triangle in $G_i$. Denote
\begin{align}\label{cN*}
 \cN_i^* = \big\{G_i \in \cN_i : v_i\in V(T) \,\text{ for some }\, T \in \cK_3(G_i) \big\}.  
\end{align}
The following lemma bounds the probability that $v_i$ is contained in a triangle without $G_i$ having a vertex with unusually large backwards degree. 


\begin{lemma}\label{diff1'}
For each $1 \leq i \leq n$ and $F\in\cN_{i-1}$
\begin{align*}
\Prob\big( G_i \in \cN^*_i \,\big|\, G_{i-1}=F \big) \le 3n^2q^3
\end{align*} 
\end{lemma}


\begin{proof}
Observe that if $G_{i-1}=F \in \cN_{i-1}$, then $e(G_{i-1}) \le 3n^2q$. It follows that, given $G_{i-1}$, the expected number of triangles in $G_i$ containing $v_i$ is at most $3n^2q^3$. Since this holds for any $G_{i-1} \in \cN_{i-1}$, the claimed bound follows by Markov's inequality.  
\end{proof}

It only remains to bound $|X_i - X_{i-1}|$ when $G_i \in \cN_i\setminus\cN_{i}^*$. We will do so using the following deterministic lemma. 

\begin{lemma}\label{diff2}
Let\/ $1\leq i\leq n$. If\/ $G_{i}\in\cN_i\setminus\cN_{i}^*$,\/ then 
\begin{align*}
|X_i - X_{i-1}| \leq 7n^2q^3.
\end{align*} 
\end{lemma}

In order to prove \Cref{diff2}, we need to open up the conditional expectation and rewrite the difference in a more transparent way. For each $1 \leq i\leq n$ and $F \in \cG_i$, define 
\begin{align*}
\cH_i(F) = \big\{ H \in \cG_n : H_i = F \big\}.
\end{align*}
Also, for each $F \in \cG_n$, define
\begin{align*}
\cE_i(F) = \big\{ H \in \cG_n : N_{H_j}(v_j) = N_{F_j}(v_j) \,\text{ for every }\, j \ne i \big\}
\end{align*}
to be the collection of graphs that agree with $H$ except for the neighbourhood of $v_i$ in $H_i$. 
Then we can partition $\cH_{i-1}(F)$ into the union of $\cE_i(F')$ over all $F'\in\cH_i(F)$.
\begin{align}\label{partition}
\cH_{i-1}(F)=\bigcup_{F'\in\cH_i(F)} \cE_i(F')
\end{align}

Now with the partition \eqref{partition}, we can open up the condition expectations clearly. Given $F \in \cG_j$ with $j > i$, we define 
$$\Prob_i(F) = \Prob\big( G_j = F \,\big|\, G_i = F_i \big)$$ 
to be the conditional probability of the event $G_j = F$, given $G_i = F_i$. And in particular for $F\in\cG_n$, let 
$$\Prob_{\cE_i}(F) = \Prob\big( G = F \,\big|\, F \in \cE_i(G) \big)$$ 
be the conditional probability of the event $G = F$, given that all of the edges except those backwards from $v_i$ agree with $F$. 

\begin{lemma}\label{lem:differences}
Given $F\in\cG_n$, for every $1\leq i\leq n$ we have
\begin{align}\label{eq:importantdifference}
|X_i(F_i)-X_{i-1}(F_{i-1})|
&\leq\sum_{F'\in\cH_{i}(F)}\sum_{F''\in\cE_{i}(F')}|s(F')-s(F'')|\cdot\Prob_i(F')\cdot\Prob_{\cE_i}(F'').
\end{align}   
\end{lemma}
\begin{proof}
We first open up the conditional expectations $X_i(F_i)$ and $X_{i-1}(F_{i-1})$ respectively, then bound the difference. Recall by definition $X_{i}(F_{i})$ is the expectation of $s(F')$ over all $F'\in\cH_{i}(F)$
\begin{align}\label{eq:rrr1}
X_{i}(F_{i})=\E\big[ s(G) \,\big|\, G_i = F_{i} \big] = \sum_{F'\in\cH_{i}(F)}s(F')\cdot\Prob_i(F'). 
\end{align}
For $X_{i-1}(F_{i-1})$, we use the partition \eqref{partition} to open up $\cH_{i-1}(F)$,
\begin{align*}
X_{i-1}(F_{i-1})=\E\big[ s(G) \,\big|\, G_{i-1} = F_{i-1} \big] &=\sum_{F''\in\cH_{i-1}(F)}s(F'')\cdot\Prob_{i-1}(F'')\\
&=\sum_{F'\in\cH_{i}(F)}\sum_{F''\in\cE_{i}(F')}s(F'')\cdot\Prob_{\cE_i}(F'') \cdot\Prob_i(F'')
\end{align*}
The last equality holds since by definition $\Prob_{\cE_i}$ is the conditional probability of $\Prob_{i-1}$ given the edges incident from $\{v_{i+1},...,v_{n}\}$, which means
\begin{align}\label{eq:tt1}
\Prob_{i-1}(F'')=\Prob_{\cE_i}(F'')\cdot\Prob_i(F'').    
\end{align}
Notice $F''\in\cE_i(F')$, by definition we have
\begin{align}\label{eq:tt2}
\Prob_i(F'')=\Prob_i(F').    
\end{align}
Then
\begin{align}\label{eq:rrr2}
X_{i-1}(F_{i-1})=\sum_{F'\in\cH_{i}(F)}\left(\sum_{F''\in\cE_{i}(F')}s(F'')\cdot\Prob_{\cE_i}(F'') \right)\cdot\Prob_i(F').
\end{align}

Combine the two decomposition above \eqref{eq:rrr1} and \eqref{eq:rrr2},
\begin{align*}
|X_i(F_i)-X_{i-1}(F_{i-1})|&\leq\sum_{F'\in\cH_{i}(F)}\left(\Big|s(F')-\sum_{F''\in\cE_{i}(F')}s(F'')\cdot\Prob_{\cE_i}(F'')\Big|\right)\cdot\Prob_i(F')
\end{align*}
Since $\sum_{F''\in\cE_{i}(F')}\Prob_{\cE_i}(F'')=1$, using triangle inequalities we can write
\begin{align*}
\Big|s(F')-\sum_{F''\in\cE_{i}(F')}s(F'')\cdot\Prob_{\cE_i}(F'')\Big|&=\Big|\sum_{F''\in\cE_{i}(F')}s(F')\cdot\Prob_{\cE_i}(F'')-\sum_{F''\in\cE_{i}(F')}s(F'')\cdot\Prob_{\cE_i}(F'')\Big|\\
&\leq\sum_{F''\in\cE_{i}(F')}|s(F')-s(F'')|\cdot\Prob_{\cE_i}(F'').
\end{align*}
Therefore substituting this into the above inequality we get
\begin{align*}
|X_i(F_i)-X_{i-1}(F_{i-1})|&\leq\sum_{F'\in\cH_{i}(F)}\sum_{F''\in\cE_{i}(F')}|s(F')-s(F'')|\cdot\Prob_i(F')\cdot\Prob_{\cE_i}(F'')
\end{align*}  
as required.
\end{proof}

Note that $|s(F)-s(F')|\in\{0,1\}$, since $F$ and $F'$ only differ on the edges between $v_i$ and $\{v_1,...,v_{i-1}\}$. We have the following necessary condition for $|s(F)-s(F')|\neq 0$.

\begin{lemma}\label{lem:nessconditionneq0}
Let $F\in\cG_n$ be such that $F_i\in\cN_i\setminus\cN_i^*$, and let $F'\in\cE_i(F)$. If  
$$|s(F)-s(F')|\neq 0,$$ 
then there either exists a triangle $(v_k,v_i,v_l)\in\cK_3(F)$ with $k<i<l$ or a triangle $(v_k,v_i,v_l)\in\cK_3(F')$ with $k<i$.
\end{lemma}

\begin{proof}
Let $\hat{F}\in\cE_i(F)$ be the graph such that there is no edge  from $v_i$ to $\{v_1,...,v_{i-1}\}$, i.e.,
$$N_{\hat{F}_{i-1}}(v_i)=\emptyset,$$
and observe that for arbitrary $F'\in \cE_i(F)$,
$$s(\hat{F})\leq s(F')\leq s(\hat{F})+1.$$

We claim that if  $s(F')=s(\hat{F})+1$, then there exists a triangle $(v_k,v_i,v_l)\in\cK_3(F')$ for some $k<i$. Since $\hat{F}\subset F'$ and $\hat{F}$ and $F'$ differ only in the edges between $v_i$ and $\{v_1,...,v_{i-1}\}$, there must be a triangle containing one such edge $v_kv_i$ ($k<i$) in $F'$ and not in $\hat{F}$. 

Note that if $|s(F')-s(F)|\neq0$, then either $s(F)\neq s(\hat{F})$ or $s(F')\neq s(\hat{F})$. Therefore, there exists a triangle $(v_k,v_i,v_l)\in\cK_3(F)\cup\cK_3(F')$ for some $k<i$. Further since 
$$F_i\notin\cN_i^*$$
which means there is no triangle containing $v_i$ in $F_i$, there could only be triangle $(v_k,v_i,v_l)\in\cK_3(F)$ with $k<i<l$.
\end{proof}

Using Lemmas~\ref{lem:differences} and~\ref{lem:nessconditionneq0}, we can now deduce \Cref{diff2}. 


\begin{proof}[Proof of \Cref{diff2}]
For every $F\in\cG_n$ such that $F_i\in\cN_i\setminus\cN_i^*$,  we define
\begin{align*}
\cM(F)=\Big\{(F',F''): |s(F')-s(F'')|=1,\, 
F'\in\cH_i(F),\, F''\in\cE_{i}(F') \Big\}.    
\end{align*}
According to the expression in \eqref{eq:importantdifference} we can write
\begin{align}\label{cM}
|X_i(F_i)-X_{i-1}(F_{i-1})|\leq\sum_{(F',F'')\in\cM(F)}\Prob_i(F')\cdot\Prob_{\cE_i}(F'').  
\end{align}
By \Cref{lem:nessconditionneq0}, for every $(F',F'') \in \cM(F)$, we have either
$$F' \in \cL_1(F) =\big\{F'\in\cH_{i}(F) : \exists \, k<i<l,\ \text{s.t.}\ (v_k,v_i,v_l)\in\cK_3(F') \big\}$$
or 
$$F''\in \cL_2(F) = \big\{F''\in\cH_{i-1}(F): \exists \, k<i,\ \text{s.t.}\ (v_k,v_i,v_l)\in\cK_3(F'') \big\}.$$
To bound the terms in~\eqref{cM} for which $F' \in \cL_1(F)$, note that
\begin{align*}
 \sum_{F''\in\cE_{i}(F')}\Prob_{\cE_i}(F'')=1  
\end{align*}
for each $F'\in\cL_1(F)$, and therefore
$$\sum_{F'\in\cL_1(F)}\sum_{F''\in\cE_{i}(F')}\Prob_i(F')\cdot\Prob_{\cE_i}(F'') \leq \sum_{F'\in\cL_1(F)}\Prob_i(F')$$
For the second term, recall from~\eqref{eq:tt1} and~\eqref{eq:tt2} that 
$$\Prob_i(F')\cdot\Prob_{\cE_i}(F'') = \Prob_{i-1}(F'')$$
for every $(F',F'') \in \cM(F)$, since $F''\in\cE_i(F')$, and therefore, by~\eqref{partition}, we have
\begin{align*}
\sum_{F'\in\cH_{i}(F)}\sum_{F''\in\cL_2(F)\cap\cE_{i}(F')}\Prob_i(F')\cdot\Prob_{\cE_i}(F'')& = \sum_{F''\in\cL_2(F)\cap\cH_{i-1}(F)}\Prob_{i-1}(F'').
\end{align*}
Hence, we can bound~\eqref{cM} as follows:
\begin{align}\label{eq:tt3}
|X_i(F_i)-X_{i-1}(F_{i-1})| \leq \sum_{F' \in \cL_1(F)} \Prob_i(F') + \sum_{F''\in\cL_2(F)}\Prob_{i-1}(F'')
\end{align}



It only remains to bound the two terms in~\eqref{eq:tt3}.

\begin{claim}\label{22221}
\begin{align*}
\sum_{F'\in\cL_1(F)} \Prob_i(F')\leq 3n^2q^3. 
\end{align*}    
\end{claim}

\begin{proof}
Recall that $\cL_1(F) \subset \cH_{i}(F)$, and observe that therefore 
\begin{align*}
\sum_{F' \in \cL_1(F)}\Prob_i(F') = \Prob\big(G\in\cL_1(F) \,\big|\, G_i = F_i \big),
\end{align*}
since $\Prob_i(F') = \Prob\big( G = F' \,\big|\, G_i = F_i \big)$ for each $F' \in \cH_{i}(F)$. Moreover, we have
$$\Prob\big( G \in\cL_1(F) \,\big|\, G_i = F_i \big) \leq \sum_{\ell = i+1}^n |N_{F_{i-1}}(v_i)| \cdot q^2  \leq 3n^2q^3,$$
since $G \in\cL_1(F)$ implies that $(v_k,v_i,v_l)$ is a triangle in $G$ for some $k < i < \ell$, and the expected number of such triangles, given $G_i = F_i$, is given by the middle term. The final inequality follows since $F_i\in\cN_i$, and therefore $|N_{F_{i-1}}(v_i)|\leq 3nq$.
\end{proof}


Bounding the second term in~\eqref{eq:tt3} is also straightforward.

\begin{claim}\label{22222}
\begin{align*}
\sum_{F''\in\cL_2(F)} \Prob_{i-1}(F'')\leq 4n^2q^3. 
\end{align*}    
\end{claim}

\begin{proof}
Observe first that
\begin{align*}
\sum_{F'' \in \cL_2(F)}\Prob_i(F'') = \Prob\big(G \in \cL_2(F) \,\big|\, G_{i-1} = F_{i-1} \big),
\end{align*}
since $\cL_2(F) \subset \cH_{i-1}(F)$. To bound the right-hand side, we break into two cases. First, let 
$$\cA_1 = \big\{ F'' \in \cL_2(F) : \exists \, (v_k,v_\ell,v_i)\in\cK_3(F'') \text{ with } k < \ell < i \big\}.$$
Conditioned on the event $G_{i-1} = F_{i-1}$, the expected number of triangles containing $v_i$ and two vertices of $V(G_{i-1})$ is at most $e(G_{i-1}) \cdot q^2$. Hence, by Markov's inequality,
$$\Prob\big( G \in \cA_1 \,\big|\, G_{i-1} = F_{i-1} \big) \leq e(G_{i-1})\cdot q^2 \leq 3n^2q^3,$$
where the final inequality holds since $F_i \in \cN_i$.

If $F'' \in \cL_2(F) \setminus \cA_1$, then it must belong to the family
$$\cA_2 = \big\{ F'' \in \cL_2(F) : \exists \, (v_k,v_i,v_\ell)\in\cK_3(F'') \text{ with } k < i < \ell \big\}.$$
Note that if $i < \ell$, then none of the edges $v_kv_i$, $v_kv_\ell$ and $v_iv_\ell$ lie in $G_{i-1}$. Therefore, again by Markov's inequality, we have 
\begin{align*}
 \Prob\big( G \in \cA_2 \,\big|\, G_{i-1} = F_{i-1} \big) \leq \E\big[ |\cK_3(G)| \big] \le n^2q^3.   
\end{align*}
Since $\cL_2(F) = \cA_1 \cup \cA_2$, we deduce that
\begin{align*}
\Prob\big( G \in \cL_2(F) \,\big|\, G_{i-1} = F_{i-1} \big) \leq \sum_{i = 1}^2 \Prob\big( G \in \cA_i \,\big|\, G_{i-1} = F_{i-1} \big) \leq 4n^2q^3, 
\end{align*}
as required.
\end{proof}

Combining~\eqref{eq:tt3} with Claims~\ref{22221} and~\ref{22222}, it follows that
\begin{align*}
|X_i(F_i)-X_{i-1}(F_{i-1})|
&\leq\sum_{F'\in\cL_1(F)}\Prob_i(F')+\sum_{F''\in\cL_2(F)}\Prob_{i-1}(F'')\leq 7n^2q^3
\end{align*}
completing the proof of the lemma.
\end{proof}

Finally, putting \Cref{diff1}, \Cref{diff1'} and \Cref{diff2} together, we can bound the quadratic variation $V_n$.

\begin{proof}[Proof of \Cref{lem:V_n}]
Recall from Observations~\ref{diff1} and~\ref{obs:degrees:ok} that
$$\Prob\big( G_i \notin \cN_i \big) \le n^{-\omega(1)} \qquad \text{and} \qquad |X_i - X_{i-1}| \leq 1$$
for every $1 \le i \le n$. It will therefore suffice to show that the bound
\begin{align*}
\E\big[(X_i - X_{i-1} )^2 \,\big|\, G_{i-1} \in \cN_{i-1} \big] \leq 4n^2q^3    
\end{align*}
holds deterministically for every $1 \le i \le n$. We will in fact show that
\begin{align}\label{eq:Vn:aim}
\E\big[(X_i - X_{i-1} )^2 \,\big|\, G_{i-1} = F \big] \leq 4n^2q^3
\end{align}
for every $F \in \cN_{i-1}$. 
In order to do so, observe first that
\begin{align}\label{eq:V_n:basic}
\E\big[(X_i-X_{i-1})^2 \,\big|\, G_{i-1} = F \big] = \sum_{F'\in\cF_i(F)}|X_i(F')-X_{i-1}(F)|^2 \cdot \Prob_{i-1}(F'),
\end{align}
where
$$\cF_i(F)= \big\{ F' \in \cG_i:F'_{i-1} = F \big\}.$$
In order to bound the sum on the right-hand side, we will partition $\cF_i(F)$ into three parts.

First, observe that, by \Cref{diff1'}, we have 
\begin{align}\label{eq:oo1}
\sum_{F'\in \cN_i^{*} \cap \cF_i(F)} \Prob_{i-1}(F')
= \Prob\big( G_{i} \in \cN_{i}^* \,\big|\, G_{i-1} = F \big)\leq 3n^2q^3.    
\end{align}
Since $|X_i - X_{i-1}| \leq 1$, by \Cref{diff1}, this is also an upper bound for the terms on the right-hand side of~\eqref{eq:V_n:basic} with  $F' \in \cN_i^*$. 
Similarly, by \Cref{obs:degrees:ok}, we have  
\begin{align}\label{eq:oo3}
\sum_{F' \in \cN_i^c \cap \cF_i(F)} |X_i(F') - X_{i-1}(F)|^2 \cdot \Prob_{i-1}(F') \leq \Prob\big( \cN_i^c \,\big|\, G_{i-1} = F \big) = n^{-\omega(1)}.   
\end{align}
Finally, recall from Lemma~\ref{diff2} that if $G_{i} \in \cN_i \setminus \cN_{i}^*$, then 
\begin{align*}
|X_i - X_{i-1}| \leq 7n^2q^3,
\end{align*} 
and therefore
\begin{align}\label{eq:oo2}
\sum_{F' \in \cF_i(F) \cap\cN_i \setminus \cN_i^*} |X_i(F') - X_{i-1}(F)|^2 \cdot \Prob_{i-1}(F') \leq \big( 7n^2q^3 \big)^2,
\end{align}
since the probabilities $\Prob_{i-1}(F')$ sum to at most $1$. 




Summing the right-hand sides of \eqref{eq:oo1}, \eqref{eq:oo3} and \eqref{eq:oo2}, and recalling~\eqref{eq:V_n:basic}, it follows that
\begin{align*}
\E\big[(X_i-X_{i-1})^2 \,\big|\, G_{i-1} = F \big] \leq 3n^2q^3 + \big( 7n^2q^3 \big)^2 + n^{-\omega(1)} \leq 4n^2q^3,
\end{align*}
where the final inequality holds for all sufficiently large $n$, since $n^{-\omega(1)} \ll n^2q^3 = o(1)$.

Since $F$ was an arbitrary element of $\cN_{i-1}$, this completes the proof of~\eqref{eq:Vn:aim}. As observed above, it follows by Observations~\ref{diff1} and~\ref{obs:degrees:ok} that
\begin{align*}
\Prob\big( V_n \geq 4n^3q^3 \big) \leq \sum_{i = 1}^n \Prob\big( G_i \notin \cN_i \big) \leq n\cdot n^{-\omega(1)}=n^{-\omega(1)},
\end{align*}
as required.
\end{proof}


\medskip

\section{Non-concentration: Proof of Theorem \ref{main2'}}\label{section6}

In order to prove the non-concentration of $s(G(n,q))$, we will adapt the coupling method of Heckel~\cite{H2} (more precisely, the version given by Heckel and Riordan~\cite{HR}) to our setting. The main challenge we will face is that their method was designed for random graphs with constant edge probability, whereas we need to deal with sequences of sparse random graphs whose edge probability varies with $n$.

In this section we will write $s(n,q)$ to denote $s(G(n,q))$. Let's fix a smooth function $q(n)$ with $n^{-1}\log n \ll q(n) \ll n^{-2/3}$, and let $\varepsilon > 0$ be sufficiently small and $n_0$ be sufficiently large. For every $n\geq n_0$ we write 
$$\alpha(n)= 3 \lfloor \varepsilon n^{3/2}q(n)^{3/2} \rfloor \qquad\text{and}\qquad n'=n+\alpha(n).$$
Our main purpose is to understand the relationship between $s(n,q(n))$ and $s(n',q(n'))$ for every $n\geq n_0$. To do this, we consider a bridge random variable $s(n',q(n))$ on $G(n',q(n))$, and construct two couplings to connect it with $s(n,q(n))$ and $s(n',q(n'))$ respectively.


We will use the method of Heckel and Riordan~\cite{HR} to obtain a coupling of $G(n,q(n))$ and $G(n',q(n))$; note that the edge probability of these two random graphs is the same.

\begin{lemma}\label{lem:coupling1}
There exists a coupling of $G(n,q(n))$ and $G(n',q(n))$ such that 
\begin{align*}
\Prob\bigg(s(n',q(n))\geq s(n,q(n))+\frac{\alpha(n)}{3}\bigg)\geq\frac12
\end{align*}
for every $n\geq n_0$.
\end{lemma}

We will prove Lemma~\ref{lem:coupling1} in Section~\ref{sec:pf:lem:coupling1}, below. In order to couple the random graphs $G(n',q(n'))$ and $G(n',q(n))$, we will use sprinkling. 


\begin{lemma}\label{lem:coupling2}
There exists a coupling of $G(n',q(n))$ and $G(n',q(n'))$ such that 
\begin{align}\label{eq:coupling2}
\Prob\big(s(n',q(n))\leq s(n',q(n'))+\varepsilon\alpha(n)\big)\geq 1 - \varepsilon
\end{align}
for every $n\geq n_0$.
\end{lemma}

We will prove Lemma~\ref{lem:coupling2} in Section~\ref{sec:pf:lem:coupling2}. Using our two couplings,  Lemmas~\ref{lem:coupling1} and~\ref{lem:coupling2}, we can easily deduce that if $s(n,q(n))$ 
is concentrated on an interval of size $o(\alpha(n))$ for every $n \ge n_0$, then its typical value 
must grow linearly with $n$, and therefore (since $n^2 q^3 \ll 1$) faster than the typical number of triangles in $G(n,q(n))$. 


\begin{lemma}\label{lem:s(n)ands(n')}
For every $n\geq n_0$, suppose there exists an interval $[a_n,b_n]$ of length 
\begin{equation*}
|b_n-a_n|\leq\varepsilon \alpha(n)   
\end{equation*}
such that
\begin{equation*}
\Prob\big(s(n,q(n))\in[a_n,b_n]\big)\geq1-\varepsilon. 
\end{equation*}
Then, deterministically,
\begin{equation*}
a_{n'}-a_n\geq(1/3-5\varepsilon)(n'-n).    
\end{equation*}
\end{lemma}

\begin{proof}
Recall that $n'=n+\alpha(n)$, and note that $\alpha(n')\leq(2n)^{3/2}q(2n)^{3/2}\leq4\alpha(n)$, since $n' \le 2n$. Fix $n \geq n_0$, and let $[a_n,b_n]$ and $[a_{n'},b_{n'}]$ be the intervals given by our assumption, so
\begin{equation}\label{eq:|b_n-a_n|}
|b_n-a_n|\leq\varepsilon \alpha(n)\qquad\text{and}\qquad|b_{n'}-a_{n'}|\leq\varepsilon\alpha(n')\leq4\varepsilon\alpha(n).     
\end{equation}
and
\begin{equation}\label{eq:Prob(s(n,q(n))}
\Prob\big( s(n,q(n))\in[a_n,b_n] \big) \geq 1 - \varepsilon \qquad \text{and} \qquad \Prob\big( s(n',q(n'))\in[a_{n'},b_{n'}] \big) \geq 1 - \varepsilon. 
\end{equation}
Now, by \Cref{lem:coupling1}, we have 
\begin{align*}
\Prob\bigg( s(n',q(n))\geq s(n,q(n))+\frac{\alpha(n)}{3} \bigg) \geq \frac12. \end{align*}
Since $\Prob\big( s(n,q(n)) \ge a_n \big) \geq 1 - \varepsilon$, by~\eqref{eq:Prob(s(n,q(n))}, it follows that
\begin{align}\label{pp1}
\Prob\bigg(s(n',q(n))\geq a_{n} + \frac{\alpha(n)}{3} \bigg) \geq \frac12 - \varepsilon.
\end{align}
On the other hand, by \Cref{lem:coupling2}, 
\begin{align*}
\Prob\big(s(n',q(n))\leq s(n',q(n')) + \varepsilon\alpha(n) \big) \geq 1 - \varepsilon.
\end{align*}
Since, by~\eqref{eq:|b_n-a_n|} and~\eqref{eq:Prob(s(n,q(n))}, we have
$$s(n',q(n')) \le b_{n'} \le a_{n'} + 4\varepsilon\alpha(n)$$
with probability at least $1 - \varepsilon$, it follows that
\begin{align}\label{eq:12345}
\Prob\big(s(n',q(n))\leq a_{n'}+5\varepsilon\alpha(n)\big)\geq1-2\varepsilon.   
\end{align}

Now put \eqref{pp1} and \eqref{eq:12345} together.
Since $\varepsilon$ is small enough, these two events on $s(n',q(n))$ have non-empty intersection, i.e., 
\begin{align*}
\Prob\bigg( a_{n}+\frac{\alpha(n)}{3} \le s(n',q(n)) \leq a_{n'}+5\varepsilon\alpha(n) \bigg) > 0.
\end{align*}
Hence, deterministically,
\begin{align*}
a_{n'}-a_n\geq (1/3-5\varepsilon)\alpha(n)=(1/3-5\varepsilon)(n'-n),
\end{align*}
as required.
\end{proof}

Now we are ready to deduce \Cref{main2'}, assuming Lemmas~\ref{lem:coupling1} and~\ref{lem:coupling2}. 

\begin{proof}[Proof of \Cref{main2'}]
Suppose, for a contradiction, that for every $\varepsilon > 0$, there exists $n_0 = n_0(\varepsilon)$ such that for every $n \ge n_0$, there exists an interval $[a_n,b_n]$ of length 
\begin{equation}\label{eq:lengthanbn}
|b_n-a_n| \leq \varepsilon \alpha(n)   
\end{equation}
such that 
\begin{equation}\label{eq:ambm}
\Prob\big( s(n,q(n)) \in [a_n,b_n] \big)\geq 1 - \varepsilon.
\end{equation}
Define a sequence $\{n_k\}_{k \ge 0}$, starting with $n_0$, by setting
$$n_{k} = n_{k-1} + \alpha(n_{k-1})$$
for each $k\geq 1$, and let 
$$K = \max\big\{ k \in \N : n_k \leq 2n_0 \big\}.$$
Note that 
\begin{align}\label{eq:n_K}
n_K \geq 2n_0- \alpha(2n_0) \geq (2 - \varepsilon) n_0, 
\end{align}
since $\alpha(n) = o(n)$ and $n_0$ was chosen sufficiently large. Moreover, 
$$ \E\big[s(n_K,q(n_K))\big]\leq (2n_0)^3q(2n_0)^3 \ll n_0,$$
since $s(n,q)$ is at most the number of triangles in $G(n,q)$, and $q(n) \ll n^{-2/3}$. By Markov's inequality, it follows that 
\begin{equation}\label{eq:S1}
\Prob\big( s(n_K,q(n_K)) \ge n_0/4 \big) \le \varepsilon.
\end{equation}
On the other hand, by \Cref{lem:s(n)ands(n')} we have
\begin{align*}
a_{n_K} \geq a_{n_0} + (1/3-5\varepsilon) \sum_{k=1}^K (n_k-n_{k-1})\geq(1/3-5\varepsilon)(n_K-n_0). 
\end{align*}
By \eqref{eq:n_K}, and assuming (as we may) that $\varepsilon$ is sufficiently small, it follows that
\begin{align*}
a_{n_K}\geq(1/3-5\varepsilon)(1 - \varepsilon) n_0 \geq n_0/4.    
\end{align*}
and hence, by \eqref{eq:ambm}, we deduce that  
\begin{equation}\label{eq:S2}
\Prob\big( s(n_K,q(n_K)) <  n_0/4 \big) \le \varepsilon,
\end{equation}
which contradicts~\eqref{eq:S1}, since $\varepsilon < 1/2$. This contradiction completes the proof. 
\end{proof}



We remain to prove these two coupling arguments \Cref{lem:coupling1} and \Cref{lem:coupling2}. We will show them respectively in the following two parts.

\subsection{Proof of \Cref{lem:coupling1}}\label{sec:pf:lem:coupling1}

In this subsection we will use the method of Heckel~\cite{H2} to construct a coupling of $G(n,q(n))$ and $G(n',q(n))$. To be precise, we will prove the following lemma, which implies \Cref{lem:coupling1}. Its proof is very similar to that of~\cite[Lemma 19]{HR}, but since it does not follow from the version stated there, we will give the details for completeness. 

\begin{lemma}\label{coupling1}
For every sufficiently large $n \in \N$, and every $0 < q_* \ll n^{-1/2}$, 
there exists a coupling of $G(n,q_*)$ and $G \sim G(n+3,q_*)$ such that
\begin{align}\label{S1}
\Prob\big( s(n+3,q_*)\geq s(n,q_*) + 1 \big)\geq 1 - \E\big[ K_3(G) \big]^{-1/2}.
\end{align}
\end{lemma}

\begin{proof}
The idea is to define a new random graph, $Q$, by adding a random triangle to the random graph $G \sim G(n+3,q_*)$. To be precise, 
let $T = \{x,y,z\}$ be a uniformly random subset of $V(G)$ of size 3, and define $Q$ to be the graph with vertex set $V(G)$ and edge set
$$E(Q) = E(G) \cup \big\{ xy,xz,yz \big\}.$$
That is, $Q$ is obtained from $G \sim G(n+3,q_*)$ by adding the three edges between the vertices of the (random) set $T$. Observe that, given $Q$ and $T$, the induced subgraph $Q[V(G) \setminus T]$ has the same distribution as $G(n,q_*)$; we may therefore couple $G(n,q_*)$ and $Q$ by making them equal on $V(G) \setminus T$. Since under this coupling we have (deterministically) 
\begin{align}\label{eq:sqsqt}
s(Q)\geq s\big( Q[ V(G) \setminus T ] \big) + 1 = s(n,q_*) + 1, 
\end{align}
to complete the proof it will suffice to prove the following bound on the total variation distance between $Q$ and $G(n+3,q_*)$. 



\begin{claim}
$$d_{\textup{TV}}\big( G(n+3,q_*), Q\big)\leq \E\big[ K_3(G) \big]^{-1/2}.$$ 
\end{claim}

\begin{proof} 
Let $G \sim G(n+3,q_*)$, 
and note that 
\begin{align*} 
d_{\textup{TV}}\big( G,Q \big) = \frac{1}{2} \sum_H \big| \Prob(H) - \Prob_Q(H) \big|,
\end{align*}
where the sum is over all (labelled) graphs $H$ with vertex set $V(G)$, and $\Prob(H)$ and $\Prob_Q(H)$ denote the probabilities of the events $G = H$ and $Q = H$, respectively. It will therefore suffice to bound the difference $| \Prob(H) - \Prob_Q(H)|$ for each such graph $H$. 


To do so, note first that
$$\Prob(H)= q_*^{e(H)} \big( 1 - q_* \big)^{\binom{n+3}{2}-e(H)}.$$
To calculate $\Prob_Q(H)$, observe that we must first choose $T$ to be the vertex set of a triangle in $H$, and then choose the remaining $e(H) - 3$ edges correctly. Thus, writing $K_3(H)$ for the number of triangles in $H$, 
we have 
\begin{align}\label{eq:pqh}
\Prob_{Q}(H)=\frac{K_3(H)}{\binom{n+3}{3}}\cdot q_*^{e(H)-3} \big( 1 - q_* \big)^{\binom{n+3}{2}-e(H)}=\frac{K_3(H)}{\E[K_3(G)]}\cdot\Prob(H),
\end{align}
where 
the second equality holds since $\E[K_3(G)]=\binom{n+3}{3}q_*^{3}$. We therefore obtain
\begin{align*} 
d_{\textup{TV}}\big( G,Q \big) 
= \frac{1}{2} \sum_H\frac{|K_3(H)-\E[K_3(G)]|}{\E[K_3(G)]} \cdot\Prob(H).
\end{align*}
Next, observe that
\begin{align*}
\sum_H \big| K_3(H)-\E[K_3(H)] \big|\cdot\Prob(H)=\E\big[ | K_3(G) - \E[ K_3(G) ] | \big] \le \Var\big( K_3(G) \big)^{1/2},    
\end{align*}
since $\E[X^2] \ge \E[X]^2$ for every random variable $X$. Moreover, since $q_* \ll n^{-1/2}$, we have
$$\Var\big( K_3(G) \big) \le \E\big[ K_3(G) \big] + n^4q_*^5 \le 2 \cdot \E\big[ K_3(G) \big].$$
It therefore 
follows that
\begin{equation*}
d_{\textup{TV}}\big( G,Q \big) \leq \E\big[ K_3(G) \big]^{-1/2} 
\end{equation*}
as claimed. 
\end{proof}
 
By~\eqref{eq:sqsqt} and the claim, it follows that
\begin{align*}
\Prob\big(s(n+3,q_*)<s(n,q_*)+1\big)\leq d_{\textup{TV}}\big( G,Q \big) \leq \E\big[ K_3(G) \big]^{-1/2} 
\end{align*}     
as required. 
\end{proof}

We can now easily deduce \Cref{lem:coupling1}.

\begin{proof}[Proof of \Cref{lem:coupling1}]
Observe first that if $n \le m \le n'$ and $q = q(n)$, then $0 < q \ll m^{-1/2}$ and 
$$\E\big[ K_3\big( G(m+3,q) \big) \big] \ge 4\alpha(n)^2,$$
by our choice of $\alpha(n)$. By Lemma~\ref{coupling1}, and since $n$ is sufficiently large, it follows that there exists a coupling of $G(m,q)$ and $G(m+3,q)$ such that
\begin{align}\label{S1}
\Prob\big( s(m+3,q)\geq s(m,q) + 1 \big)\geq 1 -\frac{1}{2\alpha(n)}.
\end{align}
Since this holds for every $n \le m \le n'$, and recalling that $n' = n + \alpha(n)$, we obtain a coupling of $G(n,q)$ and $G(n',q)$ such that 
\begin{align*}
\Prob\left(s(n',q) \ge  s(n,q)+\frac{\alpha(n)}{3}\right) \ge 1 - \frac{\alpha(n)}{2\alpha(n)} = \frac{1}{2},
\end{align*}
as required.
\end{proof}


\subsection{Proof of \Cref{lem:coupling2}}\label{sec:pf:lem:coupling2}

To complete the proof of \Cref{main2'}, it remains to construct a coupling of $G(n',q(n))$ and $G(n',q(n'))$ satisfying~\eqref{eq:coupling2}. 
The idea is to use sprinkling, and the main step in the proof is the following lemma.

\begin{lemma}\label{lem:DeltaH}
Let $H$ be a graph on $n$ vertices with $\Delta(H) \le d$, let $0 < x \ll d^{-1/2} n^{-1}$, and let $G = H \cup G(n,x)$. Then
$$|\cK_3(G) \setminus \cK_3(H)| \le \eps^3 d^{3/2}$$ 
with high probability.
\end{lemma}

\begin{proof}
We will bound separately the expected number of triangles using zero, one and two edges of $H$. First, the expected number of triangles using zero edges of $H$ is at most
$${n \choose 3} x^3 \ll d^{3/2}$$
since $x \ll d^{1/2} n^{-1}$. Second, the expected number of triangles using one edge of $H$ is at most
$$e(H) \cdot x^2n \le \Delta(H) \cdot x^2 n^2 \le d x^2 n^2 \ll d^{3/2}$$
since $x \ll d^{1/4} n^{-1}$. Finally, the expected number of triangles using two edges of $H$ is at most
$$n \Delta(H)^2 x \le d^2 x n \ll d^{3/2}$$
since $x \ll d^{-1/2} n^{-1}$. By Markov's inequality, it follows that
$$\Prob\big( |\cK_3(G) \setminus \cK_3(H)| \ge \eps^3 d^{3/2} \big) \to 0$$
as $n \to \infty$, as required.  
\end{proof}

Now let's go back to $G(n',q(n))$ and $G(n',q(n'))$. Recall from Definition~\ref{def:smoothfunction} that, since $q(n)$ is smooth and $n'=n + \alpha(n)$, we have
\begin{align}\label{eq:qn-qn'}
|q(n) - q(n')| = o\bigg(\frac{\alpha(n)}{q(n)^2n^3}\bigg) = o(n^{-3/2} q(n)^{-1/2}).   
\end{align}
In particular, since $q(n) \gg n^{-1} \log n$, we have 
\begin{align}\label{eq:qn-qn':cor}
q(n') = \big( 1 + o(1) \big) q(n).
\end{align}
We can therefore bound the typical maximum degree of $G(n',q(n'))$ by $2nq(n)$. 

\begin{observation}\label{obs:dGnq}
With high probability, 
\begin{align}\label{eq:dgnq}
\Delta\big( G(n',q(n')) \big) \le 2nq(n).  
\end{align}
\end{observation}

\begin{proof}
By Chernoff's inequality and~\eqref{eq:qn-qn':cor}, each vertex of $G(n',q(n'))$ has degree greater than $2nq(n)$ with probability at most $e^{-\Omega(nq)} \le n^{-2}$, since $q(n)\gg n^{-1}\log n$. The observation now follows by the union bound.  
\end{proof}

We are now ready to prove \Cref{lem:coupling2}.  

\begin{proof}[Proof of \Cref{lem:coupling2}]
Observe first that if $q(n) \le q(n')$, then there is nothing to prove. We may therefore assume that $q(n) \ge q(n')$. Define 
\begin{align*}
x = \frac{q(n)-q(n')}{1-q(n')}. 
\end{align*}
and observe that
\begin{align}\label{eq:union}
G(n',q(n)) \sim G(n',q(n'))\cup G(n',x),
\end{align}
and that, by~\eqref{eq:qn-qn'},
\begin{align*}
x \ll n^{-3/2} q(n)^{-1/2} = \Theta\big(  d^{-1/2}n^{-1} \big),    
\end{align*}
where $d = 2nq(n)$. We couple $G(n',q(n))$ and $G(n',q(n'))$ using~\eqref{eq:union}; that is, we first reveal $G(n',q(n'))$, and then sprinkle a copy of $G(n',x)$ to form $G(n',q(n))$. 

Now, by \Cref{obs:dGnq}, the random graph $H \sim G(n',q(n'))$ satisfies $\Delta(H) \le d$ with high probability. Moreover, given $H$, by Lemma~\ref{lem:DeltaH} we have
$$s(G) - s(H) \le |\cK_3(G) \setminus \cK_3(H)| \le \eps^3 d^{3/2}$$ 
with high probability, where $G \sim G(n',q(n))$, since we can form a triangle matching in $H$ from a triangle matching in $G$ by deleting the triangles that are not in $H$. We have therefore constructed a coupling of $G$ and $H$ such that
\begin{align*}
|s(n',q(n))-s(n',q(n'))| \leq \eps^3 \big( 2 n q(n) \big)^{3/2} \leq \eps\alpha(n),  
\end{align*}
with high probability since $\varepsilon$ is sufficiently small, as required.
\end{proof}

\medskip

\section{Further Discussions}\label{section7}

It is natural to ask whether our main results can be extended to random graphs that are somewhat less dense. In particular, can our approach be adapted to prove the conjecture of Surya and Warnke~\cite{SW}, that if
$$n^{-2/(r-1)} (\log n)^{1/{r-1 \choose 2}} \ll 1 - p \ll n^{-2/r}$$
then $\chi(G(n,p))$ is concentrated on an interval of size $O(\sqrt{\mu_r})$, or to show that $\chi(G(n,p))$ is not concentrated on any interval of size $o(\sqrt{\mu_r})$ for infinitely many values of $n \in \N$?


The main difficulty we expect to face in extending Theorems~\ref{main1} and~\ref{main2} to larger values of $r$ is in proving a structural theorem like~\Cref{structure}. We believe that such a result should be true, and make the following conjecture. For each $r \in \N$, define $S_r(G)$ to be the set of vertices of the largest $K_r$-matching in $G$ (as before, choosing according to some arbitrary deterministic rule if there is more than one). 

\begin{conjecture}\label{conj:structure}
Fix $r \geq 4$, and let $G \sim G(n,q)$, where
$$n^{-2/(r-1)} (\log n)^{1/{r-1 \choose 2}} \ll q = q(n) \ll n^{-2/r}.$$ 
With high probability as $n \to \infty$, the graph $G - S_r(G)$ contains a $K_{r-1}$-matching that covers all but $o(\sqrt{\mu_r})$ vertices of $V(G) \setminus S_r(G)$. 
\end{conjecture}



Note that~\Cref{structure} implies the case $r = 3$ of Conjecture~\ref{conj:structure}, since $\mu_3 \gg 1$ when $q \gg n^{-1}$. While it does not seem possible to extend the proof of~\Cref{structure} to larger values of $r$, we do believe that our results on the concentration of the size of the largest triangle-matching, Theorems~\ref{main1'} and~\ref{main2'}, can be generalized to $K_r$-matchings for all $r \geq 4$ using the methods of Sections~\ref{section5} and~\ref{section6}. We plan to return to this question in future work.

Another natural question is whether~\Cref{structure} can be used to prove non-concentration for $\chi(G(n,p))$ for \emph{every} $n \in \N$ in the range $n^{-1} \log n \ll 1 - p \ll n^{-2/3}$. 
Indeed, it seems plausible to us that~\Cref{thm:main*}, our central limit theorem for triangle-matchings, could be extended to the entire range via an inclusion-exclusion argument, by applying  Ruci\'nski's theorem to control the number of copies of each graph that is the union of a bounded number of triangles. However, making such an argument precise appears to require new ideas. 


Finally, it would also be interesting to understand the concentration of $\chi(G(n,p))$ in the range $n^{-1} \ll 1 - p \ll n^{-1} \log n$, which is not covered by our results or those of~\cite{SW}. Here we expect the typical deviation of $\chi(G(n,p))$ to be dominated by the  fluctuations of the number of isolated vertices in $G(n,q)$, which suggests the following conjecture.


\begin{conjecture}\label{conj:log}
If $n^{-1} \ll 1 - p \ll n^{-1} \log n$, then the typical deviation of $\chi(G(n,p))$ from its expectation is of order $\Theta\big( \sqrt{n} e^{-nq/2} \big)$, where $q = 1 - p$.
\end{conjecture}

Similarly, it seems plausible that for every $r \ge 2$, the concentration of $\chi(G(n,p))$ in the range $n^{-2/r} \ll q = 1 - p \ll n^{-2/r} (\log n)^{1/{r-1 \choose 2}}$ is controlled by the typical deviation of the number of vertices of $G(n,q)$ that are not contained in any copy of $K_r$.


\section*{Acknowledgements}
I am very grateful to my supervisor Rob Morris for many helpful discussions on the proofs and presentation of this paper.

\medskip
\pagebreak

\bibliography{bibliography}
\bibliographystyle{alpha}

\medskip
\pagebreak
\appendix

\section{Proof of Lemma \ref{regular}}

For the reader's convenience, we recall the statement of \Cref{regular}.

\begin{lemma}\label{lem:appendix}
Let $G\sim G(n,q)$, where $n^{-1}\log n\ll q\ll n^{-2/3}$, and let $\omega_0=\omega_0(n)$ be a function satisfying $\omega_0 \rightarrow \infty$ as $n \rightarrow \infty$ and $q \geq \omega^3_0 n^{-1}\log n$. Let $\cR$ be the event that $G$ has the following four properties:
\begin{itemize}
\item[$(i)$] Let $X_3$ be the number of triangles in $G$, then
$$X_3\leq n^3q^3.$$
\item[$(ii)$] For all $x,y\in V(G)$, 
$$e(N(x))\leq\omega_0\log n \qquad \textup{and} \qquad |N(x)\cap N(y)| \leq \omega_0.$$
\item[$(iii)$] For every $T\subset V(G)$ with $|T|\leq1/q$, 
$$\frac{nq|T|}{2} \leq |N(T)| \leq \frac{3nq|T|}{2}.$$
\item[$(iv)$] For every $T\subset V(G)$ with $1 \le q|T| \le \omega_0$,
$$|N(T)| \geq \big( 1 - e^{-q|T|/2} \big)n.$$
\end{itemize}

Then
\begin{align*}
 \Prob(\cR^c)=n^{-\omega(1)}.
\end{align*}
\end{lemma}

To prove part $(i)$, we will need the following result of Kim and Vu~\cite{KV} and Janson, Oleszkiewicz and Ruciński~\cite{JOR} on the upper tail of the number of triangles in $G(n,p)$. 

\begin{theorem}[Kim and Vu; Janson, Oleszkiewicz and Ruciński]\label{KVJOR}
Let $G\sim G(n,q)$ with $q > n^{-1} \log n$ and let $X$ be the number of triangles in $G(n,q)$. For any $\delta>0$,
\begin{align*}
 \Prob\big( X >(1 + \delta)\E[X]\big)< \exp\big( - c(\delta) n^{2}q^{2} \big) 
\end{align*}
for some $c(\delta) > 0$.
\end{theorem}

We remark that essentially optimal bounds on the upper tail were obtained recently by Harel, Mousset and Samotij~\cite{HMS}. 

\begin{proof}[Proof of Lemma~\ref{lem:appendix}]
Part {\rm(\romannumeral1)} follows immediately from Theorem~\ref{KVJOR}, since $\E[X_3] \le n^3q^3/2$ and $n^2q^2 \gg (\log n)^2$. For part {\rm(\romannumeral2)}, fix an arbitrary $x \in V(G)$, and reveal the edges of $G$ in two stages. First, we reveal $N(x)$, and observe that, by Chernoff's inequality, 
\begin{align}\label{eq:|N(x)|geq2nq}
\Prob(|N(x)|\geq 2nq) \leq e^{-\Theta(nq)} = n^{-\omega(1)}    
\end{align}
since $nq \gg \log n$. We then reveal the remaining edges of $G$, and observe that if $|N(x)| \leq 2nq$, then the expected number of edges in $N(x)$ is at most $2n^2q^3 \ll \omega_0\log n$. Therefore by Chernoff's inequality,
\begin{align}\label{eq:e(G[A])geqomega_0}
\Prob\big( e(N(x)) \geq \omega_0\log n \,\,\big|\,\, |N(x)| \leq 2nq \big) \leq e^{-\Theta(\omega_0\log n)} = n^{-\omega(1)}.    
\end{align}
Combining~\eqref{eq:|N(x)|geq2nq} and~\eqref{eq:e(G[A])geqomega_0}, and using a union bound over $x \in V(G)$, we deduce that 
\begin{align*}
\sum_{x \in V(G)} \Prob\big( e(N(x)) \geq \omega_0\log n \big) = n^{-\omega(1)}.
\end{align*}
For the second inequality in {\rm(\romannumeral2)}, first fix a pair of vertices $x,y\in V(G)$. By Markov's inequality, for every integer $k>0$,
\begin{align*}
\Prob\big(|N(x)\cap N(y)|\geq k\big)\leq\binom{n}{k}q^{2k}\leq (nq^2)^k.
\end{align*}
Setting $k=\omega_0$ and taking a union bound over all pairs $x,y\in V(G)$, it follows that
\begin{align*}
\Prob\Big(\bigcup_{x,y\in V(G)}|N(x)\cap N(y)|\geq \omega_0\Big)\leq n^2 \cdot (nq^2)^{\omega_0} = n^{-\omega(1)},
\end{align*}
as required. Finally, for parts {\rm(\romannumeral3)} and {\rm(\romannumeral4)}, note that for each $T \subset V(G)$, we have 
$$|N(T)| \sim \text{Bin}\big( n - |T|, 1 - (1-q)^{|T|} \big).$$
The claimed inequalities therefore follow by Chernoff's inequality.  
\end{proof}




\end{document}